\documentclass[11pt]{amsart}
\usepackage{comment}
\usepackage{geometry}
\usepackage{amsmath,amsthm}
\usepackage{amssymb}
\usepackage{xcolor}
\usepackage[colorlinks=true]{hyperref}
\hypersetup{urlcolor=blue, citecolor=red}
\usepackage{enumerate}
\usepackage{mathtools}
\usepackage{lipsum}

\usepackage[normalem]{ulem} 

\newtheorem{theorem}{Theorem}[section]
\newtheorem{corollary}{Corollary}
\newtheorem{lemma}[theorem]{Lemma}
\newtheorem{proposition}{Proposition}

\newtheorem{definition}[theorem]{Definition}

\newtheorem{remark}{Remark}

\usepackage{bbm}
\def\remspace{\kern-0.5em}

\newcommand{\CP}[1]{{\color{red}#1}}

\newcommand{\supp}{\mathrm{supp}}

\renewcommand{\d}{\mathrm{d}}
\newcommand{\ddt}{\frac{\d}{\d t}}
\newcommand{\sign}{\mathrm{sign}}
\newcommand\norm[1]{\left\lVert#1\right\rVert}

\newcommand{\ci}{c_\gamma^{(i)}}
\newcommand{\cone}{c_\gamma^{(1)}}
\newcommand{\ctwo}{c_\gamma^{(2)}}

\newcommand{\conex}{\frac{\partial c_\gamma^{(1)}}{\partial x}}
\newcommand{\ctwox}{\frac{\partial c_\gamma^{(2)}}{\partial x}}

\newcommand{\none}{n_\gamma^{(1)}}
\newcommand{\ntwo}{n_\gamma^{(2)}}

\newcommand{\px}{\dfrac{\partial p_\gamma}{\partial x}}
\newcommand{\pxx}{\dfrac{\partial^2 p_\gamma}{\partial  x^2}}

\newcommand{\p}{p_\gamma}
\newcommand{\n}{n_\gamma}
\newcommand{\w}{w_\gamma}

\newcommand{\e}{\varepsilon}

\newcommand{\partialt}[1]{\dfrac{\partial#1}{\partial t}}
\newcommand{\partialx}[1]{\dfrac{\partial#1}{\partial x}}
\newcommand{\fpartial}[1]{\dfrac{\partial}{\partial #1}}
\newlength\eqnspace
\newcommand{\R}{\mathbb{R}}
\renewcommand{\neg}[1]{\left|#1\right|_-}
\newcommand{\pos}[1]{\left|#1\right|_+}

\newcommand{\abs}[1]{\left\lvert#1\right\rvert}

\usepackage[foot]{amsaddr}

\usepackage{ulem}

\title[Hele-Shaw limit for a system of two reaction-(cross-)diffusion equations]{Hele-Shaw limit for a system of two reaction-(cross-)diffusion equations for living tissues}
\author{Federica Bubba$^\star$}
\author{Beno\^it Perthame$^\star$}
\address
{$^\star$Sorbonne Universit\'{e}, CNRS, Universit\'{e} Paris-Diderot SPC, Inria MAMBA Team, Laboratoire Jacques-Louis Lions, 4, pl. Jussieu, 75005 Paris, France.}

\author{Camille Pouchol$^\dag$}
\author{Markus Schmidtchen$^\ddag$}
\address{$^\dag$Department of Mathematics,
KTH - Royal institute of Technology}
\address{$^\ddag$Department of Mathematics, Imperial College London, SW7 2AZ London, UK}

\email[A1,A2]{federica.bubba@upmc.fr, benoit.perthame@upmc.fr, pouchol@kth.se}
\email[A1,A2]{m.schmidtchen15@imperial.ac.uk}

\begin{document}
\maketitle

\begin{abstract}
Multiphase mechanical models are now commonly used to describe living tissues including tumour growth. The specific model we study here consists of two equations of mixed  parabolic and hyperbolic type which extend the standard compressible porous medium equation, including cross-reaction terms. We study the  incompressible limit, when the pressure becomes stiff, which generates a free boundary problem. We establish  the complementarity relation and also a segregation result. 

Several major mathematical difficulties arise in the two species case. Firstly, the system structure makes comparison principles fail. Secondly, segregation and internal layers limit the regularity available on some quantities to BV. Thirdly, the Aronson-B\'enilan estimates cannot be established in our context. We are lead, as it is classical, to add correction terms. This procedure requires technical  manipulations based on BV estimates only valid in one space dimension. Another novelty is to establish an $L^1$ version in place of the standard upper bound.  
\end{abstract}

\noindent{\makebox[1in]\hrulefill}\newline

\section{Introduction}
Models of living tissues lead to the following evolution system that we consider in one space dimension, $x\in \R, \; t \geq 0$,
\begin{align}
\label{System}
\begin{cases}
\partialt \none  \remspace &= \fpartial x \left(\none\, \px\right) + \none F_1(\p) + \ntwo G_1(\p), \quad 
\\[\eqnspace]
\partialt \ntwo \remspace &= \fpartial x \left(\ntwo\, \px\right) + \none F_2(\p) + \ntwo G_2(\p),
\end{cases}
\end{align}
where the pressure $\p$ is given by the  law of state 
\begin{equation}
\p:=\big(\n \big)^\gamma, \quad \gamma>1, \qquad  \quad \n:=\none+\ntwo .
\end{equation}
We equip this system  with nonnegative initial data with compact support
\begin{equation}
\label{Initial}
n_\gamma^{(i)}(0,x) = n_{\gamma, \rm{init}}^{(i)}(x) \in L^1(\R), \; i=1,2.
\end{equation}
Here $\none,\, \ntwo$ denote the population densities, while $F_i, \; G_i$ model the reaction or growth  phenomena, which are assumed to depend exclusively on the pressure~$\p$ according to the observations in \cite{ByDr, RJPJ}. Throughout, we will call $G_1$ and $F_2$ the \textit{cross-reaction} terms. 

The purpose of our analysis is to study the Hele-Shaw limit for the solutions $\none$ and $\ntwo$ of system \eqref{System} and the pressure, namely their convergence when $\gamma$ tends to~$+\infty$. We establish both the compactness argument in order to pass to the limit in the nonlinear terms and  the limiting equations satisfied by the pressure and the densities.
\\[2mm]

\textbf{Motivation and earlier works.} 
Compressible mechanical models such as~\eqref{System} and variants have become ubiquitous in mathematical biology with applications in  modelling living tissues and tumour growth, among others. In the latter instance,  $\none$ and $\ntwo$ represent the cancer cells and quiescent or healthy cells, respectively, with different growth rates, and possible transitions from a state to the other. A question which arises here is to know as to whether segregation effects occur  between two or more species in the absence of cross-reaction terms~\cite{BCGR,Carrillo2017, Preziosi2012}.

Models of this type have been attracting attention for many decades, starting with epidemiological models, cf. \cite{BT83}. In its current form, however in the absence of any reaction terms, Eq.~\eqref{System} was proposed in the seminal paper by \cite{GP84} and an existence and segregation theory was given in  a series of papers, see~\cite{BGH87a} and the references therein.  Reaction terms were added and studied later,  the existence and segregation of solutions was established; see~\cite{BDPM10, BHIM12}.

Quite recently, the no-vacuum assumption on the initial data could be removed and thus the existence of (segregated) solutions on bounded intervals for a wider class of initial data could be shown using tools from optimal transportation; see \cite{Carrillo2017}. A little later, another existence result in higher dimensions was proposed in \cite{Gwiazda2018} by establishing strong compactness of the pressure gradient using a Aronson-B\'enilan type estimate without commenting on segregation results which are, however, expected to be true; see the general argument in Section~\ref{Section5}.

Taking the incompressible limit $\gamma \rightarrow +\infty$ has  attracted attention in the past in the one-species case, say, if $\ntwo = 0$,
\begin{equation*}
\partialt \n   = \fpartial x \left(\n\, \px\right) + \n F_1(\p).
\end{equation*}
The motivation for such asymptotics in cancer modelling is to bridge the gap between the mechanical, compressible model, \eqref{System}, and a commonly used different class of incompressible models, see the survey papers~\cite{LowengrubFrieboes_etal2010, Roose:07}. A major mathematical interest is the relation to geometric models. These are the so-called Hele-Shaw models, where the evolution of the tumour is described through the movement of the free boundary of a domain $\Omega(t)$ occupied by the tumour. In these models, the total  cell density can take up only the values $0$ and~$1$, where $1$ corresponds to the tumour.

Such results have been obtained successively in~\cite{Perthame2014} and followed by several others~\cite{Hecht2017, MoPeu, Perthame2014b}, for a different pressure law. The typical result is that the pressure $\p$ and the density $\n$ converge strongly to the limits $p_\infty$, $n_\infty$ that satisfy the so-called \textit{complementarity relation} (in the distributional sense):
\begin{equation}
\label{Complementarity_one}
p_\infty \left(\frac{\partial^2 p_\infty}{\partial x^2}  + F_1(p_\infty)\right) = 0,
\end{equation}
together with the weak form of the Hele-Shaw free boundary problem
\begin{align}
\label{Limit_one}
\begin{cases}
\partialt {n_{\infty}} =  \fpartial x \left(n_{\infty} \partialx {p_{\infty}} \right) + n_{\infty} F_1(p_{\infty}),
\\
p_\infty(1-n_\infty) = 0.
\end{cases}
\end{align}
A remarkable property is the uniqueness of the solution despite the weak relation between $n_\infty$ and $p_\infty$. The compressible model and the Hele-Shaw description of a tumour are linked through the set $\Omega(t) := \{x, \, p_{\infty}(t,x)>0\}$ which coincides a.e. with the set  $ \{n_{\infty} = 1\}$ and may therefore be considered as the tumour; see, \textit{e.g.},~\cite{MePeQu}.

Notice that the above approach to incompressible limit is not the only one. Methods based on viscosity solutions are also well-established for these growth problems~\cite {Kpo, KTu, KPS}. One can also mention that the incompressible limit is also called ``congestion" in crowd-motion, and a recent approach is based on optimal transportation arguments~\cite{MRS1, MRS2, MRSV}.

\begin{figure}
	\centering
	\includegraphics[scale=1]{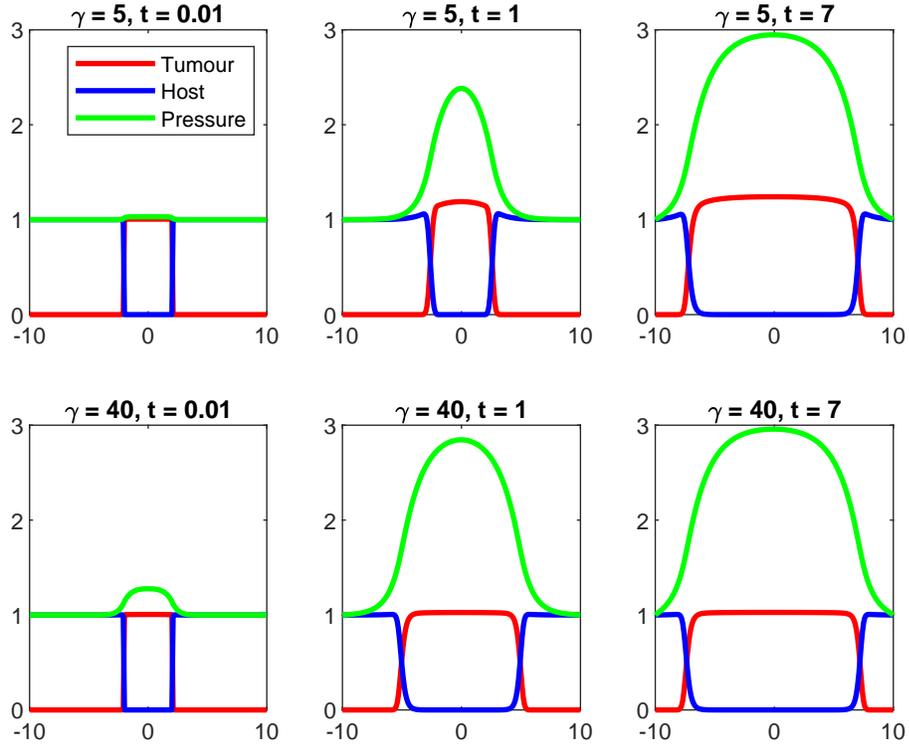}
	\vspace{-5mm}
	\caption{Comparison of evolution in time of solutions to system~\eqref{System} (blue and red line) and the corresponding pressure (green line) for different values of $\gamma$: $\gamma = 5$ (upper panels) and $\gamma = 40$ (lower panels).}
	\label{fig:comparison}
\end{figure}

\begin{figure}
	\centering
	\includegraphics[scale=0.8]{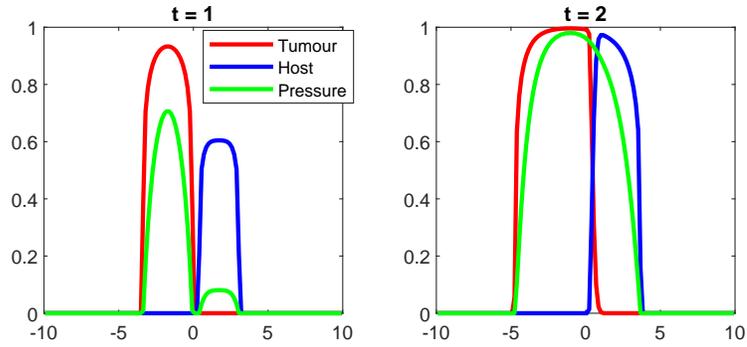}
	\vspace{-5mm}
	\caption{Solutions to~\eqref{System} and corresponding pressure in the presence of vacuum. The system is initialised with two indicator functions with supports at a positive distance. The densities remain segregated, in agreement with known results on segregated solutions (see, \textit{e.g.}, \cite{Carrillo2017}).}
	\label{fig:vacuum}
\end{figure}

The evolution of solutions to system \eqref{System} and the corresponding pressure are shown, for two values of $\gamma$, in Figure~\ref{fig:comparison}. Functions are $F_1(p) = (1-\tfrac{p}{K})$, with $K = 3$, $G_2(p) = (1-p)$ and $G_1(p) = F_2(p) = (1-p)_+$. In this case, $\none$ (tumour cells, red line) is the population with the greatest carrying capacity $K$ and it invades the other one (host cells, blue line). In Figure~\ref{fig:vacuum}, the parameters are $\gamma = 5$, $F_1(p) = 2(1-p)$ and $G_2(p) = (1-p)$, there are no cross-reaction terms, and we let the total density vanish in some parts of the domain. The two densities are initially segregated and remain so (see \cite{Carrillo2017} and Section \ref{Section5}).\\[2mm]

\textbf{Specific difficulties.} 
Two kinds of discontinuities of different natures arise in the two cell population model. The first type is  on the total cell density $n_\infty$, as in the single cell type model, generating a  free boundary moving with the Stefan condition $v=-\frac{\partial p_\infty}{\partial x}$.  The second type  are internal jumps on $n^{(i)}_\infty$ keeping $n_\infty$ continuous and thus only bounded variation is expected; cf. \textit{e.g.} \cite{BHIM12, Carrillo2017} and references therein. They are a major hurdle  because they strongly constrain the possible \textit{a priori} estimates. Also,  they are naturally related to the important segregation property mentioned earlier. Note that the paper~\cite{Degond2018}, for a different pressure law, provides a deep understanding of the internal layer dynamics in the incompressible limit when the initial data is assumed to be initially segregated with a single contact point. This has the advantage of directly applying regularity results of \cite{BDPM10}.  In this present work, we do not rely on this type of regularity. Our analysis is of a different nature being based on the regularity obtained by an Aronson-B\'enilan type estimate which is related to the method in~\cite{Gwiazda2018}. Our approach also provides two notable extensions, the cross-reaction terms and  general initial data encompassing both overlap and vacuum.  
\\
 
Aronson-B\'enilan type estimate seems necessary. In the one species case it provides uniform  $L^1$ estimates on $\Delta \p$ which are needed to establish the complementarity relation. They face two issues in the present context.  Firstly, the correct quantity has to be manipulated since  $\Delta \p$ is not enough, and the need to find the appropriate functional is standard, see~\cite{VaVi2009}. Secondly, these estimates using an upper bound by the comparison principle are not adapted to systems and here we merely work in $L^1$ for the positive part of the appropriate functional. They additionally  provide estimates useful in  the absence of good regularity for the population fractions $c_\gamma^{(i)}:= n_{\gamma}^{(i)}/\n$, which come up naturally when writing the equation satisfied by the total population: 
\begin{align*}
    \partialt {\n} = \fpartial x\left( \n \px \right) + \n \cone F(\p) + \n \ctwo G(\p),
\end{align*}
where we used the short-hand notation
\begin{align*}
F := F_1 + F_2, \quad\text{and} \quad G := G_1 + G_2.
\end{align*}

The equation for the pressure also involves the fractions and reads
\begin{align*}
\partialt \p =\left|\px\right|^2 + \gamma \p \left[ \pxx +  R \right],  \qquad  R :=\cone F(\p) + \ctwo G(\p) .
\end{align*} 

Even for a fixed $\gamma$, the population fractions are ambiguously defined whenever $\n = 0$, a scenario that typically occurs in segregation. We shall see that the population fractions do not have a well-posed limit equation, at least with the available regularity results. 

We also emphasise that this problem adds up to the more classical difficulties arising from the system structure, namely that comparison principles are not available since the interaction is neither competitive nor cooperative.
\\[2mm]

\textbf{Main results.} 
Our aim is to prove convergence for $\none$, $\ntwo$, $\n$ and $\p$ on $\mathbb{R} \times (0,T)$, for all  $T>0$, and to state the limit equations which extend~\eqref{Complementarity_one}-\eqref{Limit_one} to two species. 

As is usual for these equations of porous medium type, the solution to \eqref{System} remains compactly supported for all times whenever the initial data are (which we will assume throughout); see~\cite{Vazquez2007}.

Then, the standard idea to deal with the lack of regularity coming from $\cone$, $\ctwo$ and the free boundary, is to shift the initial data as follows
\begin{align*}
    n_{0,\e}^{(i)}(x) = n_{0}^{(i)}(x) + \e.
\end{align*}
After such a regularisation, we prove that the density $n_{\gamma, \e}$ can no longer vanish, which implies that the quotients $c^{(1)}_{\gamma, \e}$ and $c^{(2)}_{\gamma, \e}$ are well defined. Furthermore, the equation for $n_{\gamma, \e}$ (at the $\e$~level) is satisfied in the strong sense. We also recover the classical feature that uniform upper bounds hold for both densities and for the pressure.

Building on this, we are in the position to prove several \textit{a priori} estimates required to obtain enough compactness to pass to the incompressible limit and prove the complementarity relation. The key estimate is concerned with $\tfrac{\partial^2 \p}{\partial x^2}$, as it happens to be crucial to get compactness for $\tfrac{\partial \p}{\partial x}$. For this, we will adapt an argument introduced for the porous medium equation by Aronson\&B\'enilan~\cite{Aronson1979}.

The following step is to pass to the limit $\e \rightarrow 0$, $\gamma \rightarrow +\infty$ to obtain the complementarity relation and the limit equations. In doing so, we do not obtain equations for the $c_\infty^{(i)}$, but we make sure that the defining relations $c_{\gamma, \e}^{(i)} n_{\gamma, \e} = n_{\gamma, \e}^{(i)}$ remain true at the limit on $\{n_\infty>0\}$. The resulting theorem is our main result and is stated informally below. 

\begin{theorem}[Complementarity relation]
	\label{Complementarity}
With the assumptions of Section~\ref{sec:assumptions}, we may pass to the limit in Eq.~\eqref{eq:p} to obtain the complementarity relation
        \begin{align} \label{complementarityR}
	p_\infty \left[\frac{\partial^2 p_\infty}{\partial x^2}  + n_\infty^{(1)}F(p_\infty) + n_\infty^{(2)}G(p_\infty) \right] = 0.
	\end{align}
Here $n_\infty^{(i)}$, $i=1,2$  and $p_{\infty}$ weakly satisfy the equations, with $n_{\infty} = n_{\infty}^{(1)} + n_{\infty}^{(2)}$,
	\begin{equation}
	\begin{cases}
	\partialt {n_{\infty}^{(1)}} = & \fpartial x \left(n_{\infty}^{(1)} \partialx {p_{\infty}} \right) + n_{\infty}^{(1)} F_1(p_{\infty}) + n_{\infty}^{(2)} G_1(p_{\infty}),\\
	\partialt {n_{\infty}^{(2)}} = & \fpartial x \left(n_{\infty}^{(2)} \partialx {p_{\infty}} \right) + n_{\infty}^{(1)} F_2(p_{\infty}) + n_{\infty}^{(2)} G_2(p_{\infty}), \\
	\qquad 0=&p_{\infty} (1-n_{\infty}) .
	\end{cases}
	\end{equation}	 		
\end{theorem}

We also establish that segregation is preserved at the limit in the absence of cross-reactions: 
if initially $n_\infty^{(1)} n_\infty^{(2)}$ = 0, it remains true for all further times. Finally and although we do not use them here, we derive some energy estimates which are gathered in Appendix \ref{Energy} for the sake of completeness.

Note that the equality in ~\eqref{complementarityR} holds both in the distributional sense and, for a.e. $t>0$, pointwise in $x$ because we establish that $p_\infty$ is continuous in space and $ \frac{\partial^2 p_\infty}{\partial x^2} $ is a bounded measure. 

\vspace{2mm}
\textbf{Outline of the paper.}
The rest of this paper is organised as follows. In Section~\ref{Section2} we first set up the problem and state the assumptions on the reaction terms as well as the initial data, and explain why we are handling compactly supported solutions. We then introduce the regularisation by $\e$ and prove that the total density is bounded away from $0$.  
Section~\ref{Section3}  is devoted to deriving all \textit{a priori} estimates necessary for compactness. Section~\ref{Section4}  is dedicated to the incompressible limit, culminating in Theorem \ref{Complementarity}. We then tackle the problem of segregation in Section~\ref{Section5}, where the analysis is complemented by numerical simulations showing the behaviour of solutions as $\gamma \to \infty$ and the propagation of segregation. We conclude the paper with Section~\ref{Section6}, reiterating the strategy we have employed and its limitation raising several open questions.

\section{Preliminaries and regularisation}
\label{Section2}

We introduce the assumptions we need in the study of system~\eqref{System}--\eqref{Initial} when $\gamma$ tends to $\infty$.

\subsection{Assumptions and initial data}
\label{sec:assumptions}
\label{Assumptions}
\begin{definition}[Feasible data]
\label{def:feasible_growth}
	We say that the growth functions $F_i, G_i$, $i=1,2$ are feasible if they satisfy
	\begin{enumerate}[(i)]
		\item $F_i, \, G_i \in C_b^1(\R_+,\R)$, $i = 1, 2$,
		\item for any $p\geq 0$ there holds 
		\begin{align*}
			F_1'(p), \, G_2'(p) < 0, \quad \text{as well as}\quad F_2'(p), \, G_1'(p) \leq 0,
		\end{align*}
		\item there exists $P_H>0$ such that for all $p\geq P_H$:
		 \begin{align*}
			 F_1(p), \, G_2(p) \leq 0, \quad \text{ as well as } \quad F_2(p), \, G_1(p) = 0,
		 \end{align*}	
		 \item there holds
		 \begin{align*}
			 F(0) = G(0).
		 \end{align*}
	\end{enumerate}
	Throughout the paper we refer to $P_H$ as the \textbf{homeostatic pressure}.
\end{definition}
In the definition above, $C_b^1(\R_+,\R)$ is the space of $C^1(\R_+,\R)$ functions with bounded derivatives.\\
The last equality is technical but  instrumental for the Aronson-B\'enilan estimates, as in~\cite{Gwiazda2018}[Theorem 2] about the stability of weak solutions with respect to the initial data, even though it is not used in the existence result of~\cite{Carrillo2017}.
\\[2mm]

We now gather the assumptions made on the initial data. First, we assume that the initial conditions are compactly supported in some $\Omega_0$ independent of~$\gamma$, namely
\begin{equation}
\label{InitialSupport}
\supp\left(n_{\gamma, \rm{init}}^{(i)}\right) \subset \Omega_0.
\end{equation}

We define the regularised initial data for $\e > 0$ as
\begin{align*}
	n_{\gamma,\e, \rm{init}}^{(i)} = n_{\gamma,\rm{init}}^{(i)} + \e \;\; \text{on } \R, \,\; i=1,2,
\end{align*}
and we make the following set of assumptions regarding how the initial data and $\gamma$, $\e$ are related. 

\begin{definition}[Well-prepared initial data]
\label{def:initial_data}

We say the initial data are well prepared if there exist $n_{\infty,\rm init}^{(1)}$, $n_{\infty, \rm init}^{(2)}$ in $L^1(\Omega_0)$, $\e_0>0$ and $C>0$ independent of both $\gamma>1$ and $\e \leq \e_0$, such that for $i=1,2$, and all $\gamma>1$, $\e\leq \e_0$ there holds
\begin{align}
\label{eq:init}
p_{\gamma,\e_0}(0) \leq P_H, \qquad  	\lim_{\gamma \to +\infty} \norm{n_{\gamma, \rm{init}}^{(i)}-n_{\infty, \rm init}^{(i)}}_{L^1(\Omega_0)} = 0, \qquad \norm{\frac{\partial n_{\gamma, \rm{init}}^{(i)}}{\partial x}}_{L^1(\Omega_0)} \leq C,
\end{align}

\begin{align}
\label{eq:technical_w}
\norm{\frac{\partial^2 p_{\gamma, \e}}{\partial x^2}(0)}_{L^1(\Omega_0)} \leq C,
\end{align}
	
\begin{align}
\label{eq:BVcond_init}
	\left\|\frac{n_{\gamma,\e, \rm{init}}^{(i)}}{n_{\gamma,\e, \rm{init}}^{(1)}+n_{\gamma,\e, \rm{init}}^{(2)}} \right\|_{\rm{BV}(\Omega_0)} \leq C.
\end{align}
\end{definition}

The first set of conditions~\eqref{eq:init} is standard and allows to recover a density at time $0$ when passing to the incompressible limit. The second condition~\eqref{eq:technical_w} is technical and will be required when deriving some \textit{a priori} estimates in Section~\ref{Section4}.
The last set of conditions~\eqref{eq:BVcond_init}
 appears rather technical at first glance, but it is a natural assumption, cf. also~\cite{Carrillo2017}, as it allows us to handle the points where both initial densities vanish, \textit{i.e.}, vacuum or the absence of any species.

Assuming feasible data, well-prepared and compactly supported initial conditions~\eqref{InitialSupport}, we know from ~\cite{Gwiazda2018}[Theorem 3], that system~\eqref{System} admits a global weak solution $\none, \ntwo$, $p \in L^\infty(\R \times (0,T))$,  for all $T>0$. More precisely, the pressure is shown to satisfy 
\begin{equation*}
    \px \in L^2(\R \times (0,T)), \qquad     \pxx \in L^\infty(0,T;L^1(\R)),
\end{equation*}
and the weak solutions are to be understood in the following sense: for all $\phi \in C^1_{\rm comp}(\R \times (0,T))$, $i=1,2$
\begin{equation}
\label{eq:weak}
\int_0^T \int_\R \left[ -\n^{(i)} \frac{\partial \phi}{\partial t} + \n^{(i)} \px \frac{\partial \phi}{\partial x} - \left(\n^{(1)}F_i(\p)+\n^{(2)}G_i(\p) \right)\phi \right] \, \d x \d t = \int_{\R} n_{\gamma,\rm{init}}^{(i)} \phi(0) \, \d x.
\end{equation}

\subsection{Compact support}

Let us start by a few remarks on notation. Throughout we write
\begin{align*}
	|x|_+ := \begin{cases}
		x, & x > 0,\\
		0, & x \leq 0,
	\end{cases}
	\qquad \text{as well as} \qquad
	|x|_- := \begin{cases}
		-x, & x < 0,\\
		0, & x \geq 0,
	\end{cases}
\end{align*}
in order to denote the positive part of $x\in\R$ and the negative part of $x$, respectively. In particular note that then $x = |x|_+ - |x|_-$ and $|x| = |x|_+ + |x|_-$. In the same fashion  we define the positive sign and negative sign
\begin{align*}
	\sign_+(x) := \begin{cases}
		1, & x > 0,\\
		0, & x \leq 0,
	\end{cases}
	\qquad \text{as well as} \qquad
	\sign_-(x) := \begin{cases}
		-1, & x < 0,\\
		0, & x \geq 0.
	\end{cases}
\end{align*}
Note in particular that $x \cdot \sign_{\pm}(x) = |x|_{\pm}$.

Before regularising, we prove that solutions are compactly supported for all times, a result which requires checking that our assumptions ensure that both densities remain nonnegative. 

\begin{proposition}[Nonnegativity of $\none$ and $\ntwo$]
\label{prop:nonnegativity}
There holds, for all $t \geq 0$, 
	\begin{align*}
    n_{\gamma}^{(1)}(t,x) \geq 0 \quad \text{and} \quad n_{\gamma}^{(2)}(t,x) \geq 0 .
	\end{align*}
\end{proposition}
\begin{proof}
Dealing with a model as the porous medium equation, we can multiply the first equation in~\eqref{System} by $\sign_-\left(\none\right)$, to obtain
\begin{equation*}
\fpartial t \neg{\none} - \fpartial x \left( \neg{\none} \partialx{\p} \right) = \neg{\none} F_1(\p) + \ntwo  G_{1}(\p) \sign_-\left(\none\right),
\end{equation*}
by Lemma \ref{lem:krushkov} and Remark \ref{rem:krushkov}. Observe now that
\begin{align*}
	\ntwo  G_{1}(\p)\, \sign_{-}\left(\none\right) =   \left(\pos{\ntwo} -\neg{\ntwo} \right) G_{1}(\p)\, \sign_{-}\left(\none\right).
\end{align*}
Since $G_{1}(\cdot) \geq 0$, by Definition \ref{def:feasible_growth} we may write
\begin{align*}
	\ntwo  G_{1}(\p)\, \sign_{-}\left(\none\right)	&\leq -\sign_{-}\left(\none\right)\,  \neg{\ntwo} G_{1}(\p)\\
	&\leq G_{1}(0) \neg{\ntwo} , 
\end{align*}
where the last line is due to the fact that $G_1$ is decreasing in its argument, by Definition~\ref{def:feasible_growth}.
Thus, we may conclude
\begin{equation*}
\fpartial t \neg{\none} - \fpartial x \left( \neg{\none} \partialx{\p} \right) \leq \neg{\none} F_1(\p) + G_{1}(0) \neg{\ntwo}.
\end{equation*}
A similar computation for the second species yields
\begin{equation*}
\fpartial t \neg{\ntwo} - \fpartial x \left( \neg{\ntwo} \partialx{\p} \right) \leq F_2(0)\neg{\none}  + G_{2}(\p) \neg{\ntwo},
\end{equation*}
whence, upon adding both and integrating the sum in space, we obtain
\begin{equation*}
\ddt \int_{\R} \left( \neg{\none}+\neg{\ntwo} \right) \d x \leq C \int_{\R} \left( \neg{\none}+\neg{\ntwo} \right) \d x,
\end{equation*}
where $C$ only depends on the $L^\infty$-bounds of $F_i,G_i$, $i=1,2$. Applying Gronwall's lemma gives
\begin{equation*}
\int_{\R} \left( \neg{\none}+\neg{\ntwo} \right) \d x \leq 0,
\end{equation*}
and, thus, $\none(t,x) \geq 0$ and $\ntwo(t,x) \geq 0$ for all $t \geq 0$.
\end{proof}

We can now prove that solutions of~\eqref{System} are compactly supported for all times.
\begin{proposition} 
	\label{Compact}
For all $T>0$, there exists an open set $\Omega$ independent of $\gamma$ such that, 
	\begin{align*}
	\supp \, (\p(t)) \subset \Omega, \qquad \forall t \in [0,T] .
	\end{align*}
\end{proposition}

\begin{proof}
We note that we may write the pressure equation as
\begin{align*}
\partialt \p = \left|\px\right|^2 + \gamma \p \pxx + \gamma \n^{\gamma-1}\left(\none F(\p) + \ntwo G(\p)\right),
\end{align*}
and at this stage we do not  need to use the fractions $\cone$, $\ctwo$. 
Since $\none, \ntwo \geq 0$, we have
$\none F(\p) + \ntwo G(\p) \leq (\none + \ntwo) \max(F(\p),G(\p)) = \n \max(F(\p),G(\p))$.
We infer
\begin{align*}
\partialt \p & \leq \left|\px\right|^2 + \gamma \p \pxx + \gamma \n^{\gamma-1}\n \max(F(\p),G(\p))\\
 & = \left|\px\right|^2 + \gamma \p \pxx + \gamma \p \max(F(\p),G(\p)).
\end{align*}
Therefore $\p$ is a subsolution of the equation satisfied with reaction function $\max(F,G)$. For this equation, it is well known that compactly supported initial data leads to a compactly supported solution for all times~\cite{Perthame2014}. 
\end{proof}
We may without loss of generality assume that $\Omega = (-L,L)$ for some $L>0$, and we define the set $Q_T: = \Omega \times (0,T)$.  We fix these arbitrary parameters $T$ and $L$. 

\subsection{Regularisation and strong solutions}

As explained above, the regularisation step is purely technical, yet necessary, for the rigorous derivation of the \textit{a priori} estimates in the subsequent section.\\

\textbf{Regularised equations.}
\label{Regularisation}
We denote by $n_{\gamma,\e}^{(1)}, \, n_{\gamma,\e}^{(2)}$ the solutions of the system when the initial data are the regularised ones: $n_{\gamma,\e, \rm{init}}^{(i)} = n_{\gamma,\rm{init}}^{(i)} + \e$, $i=1,2$ for $\e > 0$. From now on, we  consider the system
\begin{align}
\label{eq:regularised}
\begin{cases}
		\partialt {n_{\gamma,\e}^{(1)}}  \remspace &= \fpartial x \left(n_{\gamma,\e}^{(1)}\, \partialx {p_{\gamma,\e}} \right) + n_{\gamma,\e}^{(1)} F_1(p_{\gamma, \e}) + n_{\gamma,\e}^{(2)} G_1(p_{\gamma, \e}),\\[\eqnspace]
    	\partialt {n_{\gamma,\e}^{(2)}}  \remspace &= \fpartial x \left(n_{\gamma,\e}^{(2)}\, \partialx {p_{\gamma,\e}}\right) + n_{\gamma,\e}^{(1)} F_2(p_{\gamma, \e}) + n_{\gamma,\e}^{(2)} G_2(p_{\gamma, \e}),
\end{cases}
\end{align}
with, as before, $p_{\gamma, \e} = n_{\gamma,\e}^\gamma$ with $n_{\gamma,\e} = n_{\gamma,\e}^{(1)} + n_{\gamma,\e}^{(2)}$. 

The regularised total density $n_{\gamma,\e}$ satisfies the equation
\begin{align}
\label{eq:regularised_n}
\partialt {n_{\gamma,\e}} = \fpartial x \left( n_{\gamma,\e} \, \partialx {p_{\gamma,\e} } \right) + n_{\gamma,\e}\, c_{\gamma,\e}^{(1)} \, F(p_{\gamma,\e}) + n_{\gamma,\e} \, c_{\gamma,\e}^{(2)} \, G(p_{\gamma,\e}),
\end{align}
which we endow with homogeneous Neumann boundary conditions
\begin{align*}
\dfrac{\partial n_{\gamma, \e}}{\partial x}(t,x) = 0, \; x = \pm L.
\end{align*}
Thanks to~\cite{Gwiazda2018}[Theorem 2], as $\e \rightarrow 0$, there is convergence of solutions of the regularised system towards those of the original one. More precisely, $n_{\gamma,\e}^{(i)}$ converges to $n_{\gamma}^{(i)}$ in $L^\infty(Q_T)-w\star$ for $i=1,2$, $n_{\gamma,\e}$, $p_{\gamma,\e}$ converge to $\n$, $\p$, in $L^q(Q_T), 1\leq q < \infty$. Moreover, $\partial_x  p_{\gamma,\e}$ strongly converges to $\partial_x \p$ in $L^2(Q_T)$.

As we shall now prove, the regularised total density is positive. This allows us to define the quotients $c^{(1)}_{\gamma, \e}$ and $c^{(2)}_{\gamma, \e}$. On the other hand, the positivity ensures that $n_{\gamma,\e}$ is a strong solution of \eqref{eq:regularised_n}.
In fact, the regularisation allows us to get rid of the degenerate parabolicity of the equation, and then solutions are classical~\cite{Vazquez2007}[Theorem 3.1].

Finally, the associated pressure satisfies, in the strong sense,
\begin{align}
\label{eq:p}
	\partialt {p_{\gamma, \e}} =\left|\partialx {p_{\gamma, \e}}\right|^2 + \gamma p_{\gamma, \e}  \left[ \frac{\partial^2 p_{\gamma, \e}}{\partial x^2}+ \, R_{\gamma, \e} \right] .
\end{align}
\\
From now on, we keep the regularisation parameter $\e >0$ in the statement of all propositions and theorems below while dropping it in the proofs to allow for an improved readability. 
\\

\textbf{Positivity for the density.} We now build a subsolution for the equation on $n_{\gamma, \e}$. A difficulty is that we cannot hope to derive any general comparison result at the level of the system. For that reason, the control from below relies on the observation that, thanks to the definition of $\p$, \eqref{eq:regularised_n} can be rewritten in a porous medium equation form:
\begin{equation}
\label{eq:pme}
\frac{\partial n_{\gamma,\e}}{\partial t} - \frac{\gamma}{\gamma+1}\frac{\partial^2}{\partial x^2}\left(n_{\gamma,\e}^{\gamma+1}\right) = n_{\gamma,\e} R_{\gamma, \e}.
\end{equation}

\begin{proposition}[Positivity for the density $n_{\gamma,\e}$]
\label{prop:positivity}
The solution to~\eqref{eq:regularised_n} satisfies
\begin{equation*}
n_{\gamma,\e} \geq  \underline{n}_{\gamma} := 2 \e e^{-R_{\infty}t} >0,
\end{equation*}
for $t \in (0, T]$, where $R_{\infty}>0$ is a $L^\infty$ bound for $| R_{\gamma, \e}|$. 
\end{proposition}

\begin{proof}
We have chosen $\underline{n}_{\gamma}$ so that $\frac{\partial \underline{n}_{\gamma}}{\partial t} = -R_{\infty}\, \underline{n}_{\gamma}$ and $\underline{n}_{\gamma}(0) = 2 \e \leq \n(0)$.

Subtracting the equations for $\underline{n}_{\gamma}$ and $\n$ we get
	\begin{equation*}
	\fpartial t \left(\underline{n}_{\gamma} - \n \right) - \frac{\gamma}{\gamma+1}\frac{\partial^2}{\partial x^2}\left(\underline{n}_{\gamma}^{\gamma+1}- \n^{\gamma+1} \right) = - R_{\infty}\underline{n}_{\gamma}- \n R.
	\end{equation*}
	Multiplying by $\sign_+(\underline{n}_{\gamma} - \n )$ and using that $\sign_+( \underline{n}_{\gamma}^{\gamma+1} - n_{\gamma}^{\gamma+1} ) = \sign_+(\underline{n}_{\gamma} - \n)$, we arrive at
	\begin{equation*}
	\fpartial t \left| \underline{n}_{\gamma} - \n \right|_+ - \frac{\gamma}{\gamma+1}\frac{\partial^2}{\partial x^2}\left|\underline{n}_{\gamma}^{\gamma+1} - \n^{\gamma+1}\right|_+ \leq \left( - R_{\infty} \underline{n}_{\gamma}- \n R \right)\sign_+(\underline{n}_{\gamma} - \n ).
	\end{equation*}
	We now observe that we can write
	\begin{equation*}
	\begin{split}
	- R_{\infty}\underline{n}_{\gamma}- \n R = \underline{n}_{\gamma} \left(-R_{\infty} - R \right) + R \left(\underline{n}_{\gamma}-\n\right).
	\end{split}
	\end{equation*}
	Using that the quantity $-R_{\infty} - R$ is always negative thanks to the definition of $R_{\infty}$, we obtain
	\begin{equation*}
	\fpartial t \left| \underline{n}_{\gamma} - \n \right|_+ - \frac{\gamma}{\gamma+1}\frac{\partial^2}{\partial x^2}\left|\underline{n}_{\gamma}^{\gamma+1} - \n^{\gamma+1}\right|_+ \leq R\left|\underline{n}_{\gamma}-\n\right|_+.
	\end{equation*}
	Integrating in space yieds
	\begin{equation*}
	\ddt \int_{\Omega} \left| \underline{n}_{\gamma} - \n \right|_+ \d x \leq R_\infty \int_{\Omega} \left|\underline{n}_{\gamma}-\n\right|_+ \d x,
	\end{equation*}
	which, thanks to Gronwall's lemma and the hypothesis on initial conditions, implies that a.e. $\n(t,x) \geq \underline{n}_{\gamma}(t,x)$ for $t \in (0,T]$.
\end{proof}

\begin{remark}
More generally, the previous result shows that for every nonnegative $\underline{n}_{\gamma}$ and $\bar{n}_{\gamma}$ that satisfy respectively
\begin{equation*}
\begin{split}
\frac{\partial \underline{n}_{\gamma}}{\partial t} - \frac{\gamma}{\gamma+1}\frac{\partial^2 \underline{n}_{\gamma}^{\gamma+1}}{\partial x^2} &\leq -R_{\infty}\, \underline{n}_{\gamma},
\\
\frac{\partial \bar{n}_{\gamma}}{\partial t}- \frac{\gamma}{\gamma+1}\frac{\partial^2 \bar{n}_{\gamma}^{\gamma+1}}{\partial x^2} &\geq R_{\infty}\,\bar{n}_{\gamma},
\end{split}
\end{equation*}
with $R_{\infty}>0$ as in Proposition~\ref{prop:positivity}, we have a.e. in $Q_T$
\begin{equation*}
\underline{n}_{\gamma}(t,x) \leq n_{\gamma}(t,x) \leq \bar{n}_{\gamma}(t,x).
\end{equation*}
\end{remark}
\begin{proposition}[Uniform bounds for $n_{\gamma,\e}$ and $p_{\gamma,\e}$]
The solution of \eqref{eq:regularised_n}, satisfies for a.e. $(t,x) \in Q_T$
	\begin{align*}
	0 < n_{\gamma, \e}(t,x) \leq P_H^{1/\gamma} \quad \text{and}\quad 0 < p_{\gamma, \e}(t,x) \leq P_H .
	\end{align*}
\end{proposition}

\begin{proof}
From Proposition~\ref{prop:positivity}, it is clear that $\p = (\n)^{\gamma} > 0$.\\
As for the $L^{\infty}$ bounds, we set $n_H = (P_H)^{1/\gamma}$ and observe that, for every $0 \leq \zeta, \xi \leq 1$, 
\begin{equation*}
0 \geq n_H \left( \zeta F(P_H) + \xi G(P_H) \right),
\end{equation*}
thanks to the definition of $P_H$. Thus, we can choose $\zeta = \cone$ and $\xi = \ctwo$ and say that
\begin{equation*}
\begin{split}
\fpartial t \left( \n - n_H \right) &- \frac{\gamma}{\gamma+1} \frac{\partial^2}{\partial x^2} \left( \n^{\gamma +1} - n_H^{\gamma +1} \right)\\
& = R\left( \n - n_H \right)  + n_H \left[\cone \left( F(p) - F(P_H) \right) + \ctwo \left( G(p) - G(P_H) \right) \right].
\end{split}
\end{equation*}
Multiplying by $\sign_+(\n - n_H)$ and observing that \[\left( F(\p) - F(P_H) \right) \sign_+(\n - n_H), \, \left( G(\p) - G(P_H) \right) \sign_+(\n - n_H) \leq 0\] thanks to Definition \ref{def:feasible_growth}, we get
\begin{equation*}
\begin{split}
\fpartial t \left| \n - n_H \right|_+ - \frac{\gamma}{\gamma+1} \frac{\partial^2}{\partial x^2} \left| \n^{\gamma +1} - n_H^{\gamma +1} \right|_+ \leq R\left| \n - n_H \right|_+.
\end{split}
\end{equation*}
Integrating in space and using Gronwall's lemma, we obtain
\begin{equation*}
\ddt \int_{\Omega} \left| \n - n_H \right|_+ \d x \leq C \int_{\Omega} \left| \n - n_H \right|_+ \d x,
\end{equation*}
which implies $\n(t,x) \leq n_H$ and $\p(t,x) \leq P_H$ for every $t \in (0,T]$ since it holds initially by \eqref{eq:init}.
\end{proof}

\subsection{Equations for the fractions}

Let us examine the equations satisfied by the concentrations $0 \leq c_{\gamma, \e}^{(i)} = n_{\gamma, \e}^{(i)} / n_{\gamma,\e} \leq 1$ for $i=1,2$. Thanks to the positivity of $n_{\gamma,\e}$ these quantities are well defined. Deriving the equations they satisfy then also requires the smoothness of $n_{\gamma,\e}$.

By \eqref{System} and \eqref{eq:regularised_n}, they formally satisfy
\begin{equation*}
\begin{split}
	\partialt{c^{(1)}_{\gamma,\e}} =&
	\frac{1}{n_{\gamma,\e}} \left( \partialt{n^{(1)}_{\gamma,\e}} - \frac{n^{(1)}_{\gamma,\e}}{n_{\gamma,\e}} \partialt{n_{\gamma,\e}} \right)\\
	=& \frac{1}{n_{\gamma,\e}} \fpartial x \left(n^{(1)}_{\gamma,\e} \partialx{p_{\gamma,\e}}\right) + c^{(1)}_{\gamma,\e} F_1(p_{\gamma,\e}) + c^{(2)}_{\gamma,\e} G_1(p_{\gamma,\e})\\
	& \quad - \frac{c^{(1)}_{\gamma,\e}}{n_{\gamma,\e}} \fpartial x \left(n_{\gamma,\e} \partialx{p_{\gamma,\e}}\right) - \left(c^{(1)}_{\gamma,\e}\right)^{2} F(p_{\gamma,\e}) - c^{(1)}_{\gamma,\e} \, c^{(2)}_{\gamma,\e} \, G(p_{\gamma,\e}).
\end{split}
\end{equation*}
Thus, observing that
\begin{align*}
    \displaystyle \frac{1}{n_{\gamma,\e}} \fpartial x \left(n^{(1)}_{\gamma,\e} \partialx{p_{\gamma,\e}}\right) - \frac{c^{(1)}_{\gamma,\e}}{n_{\gamma,\e}} \fpartial x \left(n_{\gamma,\e} \partialx{p_{\gamma,\e}}\right) = \partialx {c^{(1)}_{\gamma,\e}} \, \partialx {p_{\gamma,\e}},
\end{align*}
we derive the two equations for $c^{(1)}_{\gamma,\e}$ and $c^{(2)}_{\gamma,\e}$:
\begin{align}
	\label{eq:system_ratios}
	\partialt {c^{(1)}_{\gamma,\e}} = \partialx {c^{(1)}_{\gamma,\e}} \partialx{p_{\gamma,\e}} + c^{(1)}_{\gamma,\e} F_1(p_{\gamma,\e}) + c^{(2)}_{\gamma,\e} G_1(p_{\gamma,\e}) - \left(c^{(1)}_{\gamma,\e}\right)^2 F(p_{\gamma,\e}) - c^{(1)}_{\gamma,\e} c^{(2)}_{\gamma,\e} G(p_{\gamma,\e}),\\[\eqnspace]
    \partialt {c^{(2)}_{\gamma,\e}} = \partialx {c^{(2)}_{\gamma,\e}} \partialx{p_{\gamma,\e}} + c^{(1)}_{\gamma,\e} F_2(p_{\gamma,\e}) + c^{(2)}_{\gamma,\e} G_2(p_{\gamma,\e}) - \left(c^{(2)}_{\gamma,\e}\right)^2 G(p_{\gamma,\e}) - c^{(1)}_{\gamma,\e} c^{(2)}_{\gamma,\e} F(p_{\gamma,\e}).
\end{align}

Note that these equations are to be understood in the weak sense. For $c^{(1)}_{\gamma,\e}$, for example, it is given by
\begin{align*}
\label{eq:weak_fractions}
\int_0^T \int_\Omega \bigg[-c^{(1)}_{\gamma,\e} \frac{\partial \phi}{\partial t} + c^{(1)}_{\gamma,\e} \frac{\partial p_{\gamma,\e}}{\partial x} \frac{\partial \phi}{\partial x} +  c^{(1)}_{\gamma,\e} \frac{\partial^2 p_{\gamma,\e}}{\partial x^2} \phi&   \\ - \left(c^{(1)}_{\gamma,\e} F_1(p_{\gamma,\e}) + c^{(2)}_{\gamma,\e} G_1(p_{\gamma,\e}) - \left(c^{(1)}_{\gamma,\e}\right)^2 F(p_{\gamma,\e}) - c^{(1)}_{\gamma,\e} c^{(2)}_{\gamma,\e} G(p_{\gamma,\e})\right)\phi \bigg] \, \d x \d t & = \int_{\Omega} c_{\gamma, \e}^{(1)}(0)\, \phi(0) \, \d x.
\end{align*}

They are obtained by choosing $\phi / n_{\gamma,\e}$ as a test function in the weak formulation~\eqref{eq:weak} of the equations for $n^{(1)}_{\gamma,\e}$ and $n^{(2)}_{\gamma,\e}$, with $\phi$ smooth and compactly supported in $\Omega$. This choice of test function is made possible by the smoothness of $n_{\gamma,\e}$. 

\section{A Priori Estimates}
\label{Section3}

Throughout, $C$ will denote a constant independent of $\gamma$ and $\e$ (but which might depend on $T$), which may also change from line to line.
This section is dedicated to proving the following \textit{a priori} estimates. The last estimate involves 
\begin{equation}
w_{\gamma, \e}:= \frac{\partial^2 p_{\gamma, \e} }{\partial x^2} + R_{\gamma, \e}.
\end{equation}
\begin{theorem}[A Priori Estimates]
	\label{thm_apriori}
	\begin{align}
	\label{eq:estimate_ltwo}
	\int_0^T \int_\Omega \left|\frac{\partial p_{\gamma, \e}}{\partial x}\right|^2\d x \d t \leq C,
	\end{align}
\begin{align}
\label{eq:estimate_bv}
	\sup_{0 \leq t \leq T} \int_\Omega \left( \left|\partialx {c_{\gamma, \e}^{(i)}}\right|  + \left|\partialx {n_{\gamma, \e}^{(i)}}  \right| \right) \d x \leq C, \; \; i=1,2,
\end{align}

\begin{align}
\label{eq:estimate_w}
	\gamma\, \int_0^T \int_\Omega p_{\gamma, \e} |w_{\gamma, \e} |\d x \d t\leq C, \qquad 
\sup_{0 \leq t \leq T} \int_\Omega \neg{w_{\gamma, \e}} \d x\leq C .
\end{align}
\end{theorem}

The proof of the theorem is split into several results that are proven below in  chronological order.

\begin{proof} 
{\it (Estimate \eqref{eq:estimate_ltwo}).}  We integrate the equation for the pressure, Eq.~\eqref{eq:p}, in space to obtain
\begin{align*}
	\ddt \int_\Omega \p \d x- \int_\Omega \left( \left|\px\right|^2 + \gamma \p \pxx \right) \d x = \gamma \int_\Omega \p R \d x.
\end{align*}
An integration by parts in the second-order term yields
\begin{align*}
	\ddt \int_\Omega \p \d x+  (\gamma - 1) \int_\Omega \left| \px \right|^2  \d x  = \gamma \int_\Omega \p R \d x,
\end{align*}
having used the fact that $p_\gamma \frac{\partial p_\gamma}{\partial x}$ vanishes at the boundary due to the Neumann boundary conditions. Finally, let us integrate in time to get
\begin{align*}
	(\gamma-1)\int_0^T\int_\Omega \left|\px\right|^2 \d x \d t 
    =\gamma \int_0^T\int_\Omega \p R \, \d x \d t - \|\p(T)\|_{L^1(\Omega)} + \|\p(0)\|_{L^1(\Omega)} \leq C \gamma,
\end{align*}
as $\p$ and $R$ are bounded in $L^\infty(Q_T)$ uniformly in $\gamma$ and $\e$. We conclude by dividing by $\gamma$ to obtain the desired estimate.
\end{proof}

\begin{proof}
{\it (Estimate~\eqref{eq:estimate_bv}).} We begin by considering the equation for $\cone$ in Eq. \eqref{eq:system_ratios}. Upon differentiation in space, there holds
\begin{align*}
\fpartial t \conex = 
	&~ \fpartial x \left(\conex \px \right) + \conex F_1(\p) + \ctwox G_1(\p)\\
 	&- 2 \cone \conex F(\p) - \ctwo \conex G(\p) - \cone \ctwox G(\p) \\
 	&+ \px \left( \cone  F_1'(\p) + \ctwo G_1'(\p) - (\cone)^2 F'(\p) - \cone \ctwo G'(\p)\right).
\end{align*}
Multiplying this equation by $\sign(\tfrac{\partial \cone}{\partial x})$ and using Lemma \ref{lem:krushkov} and Remark \ref{rem:krushkov}, we obtain
\begin{align*}
\fpartial t \left|\conex\right| \leq &~ \fpartial x\left(\left|\conex\right|\p \right) + \left|\conex\right|\left( F_1(\p) - 2 \cone F(\p) - \ctwo G(\p) \right) \\
& +\ctwox \left(G_1(\p)   - \cone G(\p) \right) \sign\left(\conex\right)\\
& + \px \left( \cone F_1'(\p) + \ctwo G_1'(\p) - (\cone)^2 F'(\p) -\cone \ctwo  G'(\p)\right)\sign\left(\conex\right)\\
\leq &~\fpartial x\left(\left|\conex\right|\p \right) + C \left|\conex\right| + C \left|\frac{\partial c_\gamma^{(2)}}{\partial x}\right| + C \left| \px\right|,
\end{align*}
where the constants are independent of $\gamma$ and $\e$ and only depend on the $L^\infty$-bounds on $F_i,G_i$, as well as on the fact that $0\leq c_\gamma^{(i)}\leq 1$, $i=1,2$. 
Upon integrating in space we get
\begin{align*}
\ddt\int_\Omega  \left|\conex\right| \d x \leq C \int_\Omega \left( \left|\conex\right| + \left|\ctwox\right| \right) \d x + C \int_\Omega \left|\px\right| d x,
\end{align*}
where the first term has vanished as it was an exact derivative and the boundary terms vanish by the homogeneous Neumann boundary conditions. 

Performing the same manipulations on the equation for $c_\gamma^{(2)}$ and summing both, we finally get
\begin{align*}
\ddt\int_\Omega  \left(\left|\conex\right| + \left|\ctwox\right| \right) \d x \leq C \int_\Omega \left(\left|\conex\right| + \left|\ctwox\right| \right) \d x + C \int_\Omega \left|\px\right| \d x.
\end{align*}

By setting
\begin{align*}
	\psi(t):=\int_\Omega \left(\left|\conex\right| + \left|\ctwox\right| \right) \d x, 
\end{align*}
the previous inequality reads 
\begin{align*}
	\psi'(t) \leq C \psi(t) + C \int_\Omega \left|\px\right| \d x.
\end{align*}
An application of Gronwall's lemma yields 
\begin{align*}
	\psi(t) \leq C \psi(0) e^{Ct} + C \int_0^t \int_\Omega  e^{C(t-s)}  \left|\px\right| \d x \d s.
\end{align*}

From the uniform $L^2(Q_T)$-bounds on $\tfrac{\partial \p}{\partial x}$, we conclude that 
\begin{align*}
	\psi(t)\leq C \psi(0) + C.
\end{align*}

At this stage let us emphasise that none of the constants depends on $\gamma$ or $\e$. Finally, the term $\psi(0)$ is bounded by the assumptions on the initial data, cf. Eq. \eqref{eq:BVcond_init}.

For the densities $\none$, $\ntwo$, we start by estimating the total density $\n$. We differentiate Eq.~\eqref{eq:regularised_n} w.r.t. $x$, which yields
\begin{align*}
	\fpartial t \left( \partialx{n_{\gamma}}\right) =
	&  \frac{\partial^2}{\partial x^2} \left( \n \partialx {\p} \right) + 
    \n \left( \conex F(\p) + \ctwox G(\p) \right) \\ 
    &+ \partialx {n_{\gamma}} \left( \cone F(\p) + \ctwo G(\p) \right) + 
    \px \n \left( \cone F'(\p) + \ctwo G'(\p)\right).
\end{align*}

Using $\n \tfrac{\partial \p}{\partial x} = \gamma \p \tfrac{\partial \n}{\partial x}$ for the first term in the right-hand side, multiplying by $\sign(\tfrac{\partial \n}{\partial x}) = \sign(\tfrac{\partial \p}{\partial x})$ and  employing Lemma \ref{lem:krushkov} and Remark \ref{rem:krushkov}, we obtain
\begin{align*}
	\fpartial t \left| \partialx {n_{\gamma}}\right| & \leq \gamma 
     \frac{\partial^2}{\partial x^2} \left( \p \left|\partialx {n_{\gamma}}\right| \right) + 
    \n \left( \conex  F(\p) + \ctwox  G(\p)\right) \sign\left(\partialx {n_\gamma}\right) \\
    &\quad  + \left|\partialx {n_{\gamma}}\right| \left( \cone F(\p) + \ctwo G(\p) \right) + 
    \left|\px\right| \n \left( \cone F'(\p) + \ctwo G'(\p)\right).
\end{align*}
Upon integrating in space and using the zero Neumann boundary conditions, we get
\begin{align*}
	\ddt \int_\Omega \left| \partialx {n_{\gamma}}\right|\d x \leq  C\int_\Omega \left(
     \left| \conex\right| + \left|\ctwox  \right| + \left|\partialx {n_{\gamma}}\right|  + 
    \left|\px\right| \right) \d x,
\end{align*}
where the constant $C>0$ only depends on the $L^\infty$-bounds of $\n$, as well as $F_i,G_i$ and $F_i',G_i'$, for $i=1,2$. Using the $BV$-bounds from above we may further write
\begin{align*}
	\ddt \int_\Omega \left| \partialx {n_{\gamma}}\right|\d x \leq  C + C \int_\Omega 
 \left|\partialx {n_{\gamma}}\right| \d x + 
  \int_\Omega   \left|\px\right|\d x.
\end{align*}
Proceeding as before, we conclude that $\n$ is uniformly bounded in $BV$, cf. Definition \ref{def:initial_data}, Eq. \eqref{eq:BVcond_init}. To see that it provides the required $BV$ estimates for $\none$, $\ntwo$, we notice that the equality $n_\gamma^{(i)} =c_\gamma^{(i)} \n$ leads to 
\begin{align*}
\partialx {\n^{(i)}} = c_\gamma^{(i)} \partialx {\n} + \n \partialx {c_\gamma^{(i)}}
\end{align*}
for $i=1,2$. The $BV$ bounds for $\n$ and $\cone$ and $\ctwo$ together with the boundedness of $\n$ then imply the result and conclude the proof.
\end{proof}

%
%

\begin{proof}
{\it (Estimate ~\eqref{eq:estimate_w}).} 
For ease on notations, we set $R = \cone F(\p) + \ctwo G(\p)$, as before, and we recall that $R$ is bounded in $L^{\infty}(Q_T)$. Using Eqs. \eqref{eq:p} and \eqref{eq:system_ratios}, we want to differentiate the quantity $\w$ in time. We obtain
\begin{align*}
	\partialt \w = \underbrace{\; \; \fpartial t \pxx \; \;}_{I_1} + \underbrace{\; \; \partialt R\; \;}_{I_2}.
\end{align*}
We shall address both terms individually beginning with $I_2$.
\begin{align*}
	I_2 
	&= \fpartial t \left(\cone F(\p) + \ctwo G(\p) \right)\\
	&=\underbrace{\partialt \cone F(\p) + \partialt \ctwo G(\p)}_{I_{2,1}} + \underbrace{\left( \cone F'(\p) + \ctwo G'(\p)\right)\partialt \p}_{I_{2,2}}.
\end{align*}
Using the equations for $c_\gamma^{(i)}$, Eq. \eqref{eq:system_ratios}, we obtain
\begin{align*}
	I_{2,1} &= F(\p) \left(\partialx {\cone} \px + \cone F_1(\p) + \ctwo G_1(\p) - \left(\cone\right)^2 F(\p) - \cone \ctwo G(\p)\right)\\
	&\quad +G(\p)\left(\partialx {\ctwo} \px + \cone F_2(\p) + \ctwo G_2(\p) - \left(\ctwo\right)^2 G(\p) - \cone \ctwo F(\p)\right)\\
	&= \px \left( F(\p) \conex + G(\p)\ctwox\right) + S_2,
\end{align*}
where we introduced
\begin{align}
	\label{eq:Stwoequation}
	\begin{split}
	S_2 &= F(\p) \left( \cone F_1(\p) + \ctwo G_1(\p) - \left(\cone\right)^2 F(\p) - \cone \ctwo G(\p) \right)  \\
    &\quad + G(\p) \left( \cone F_2(\p) + \ctwo G_2(\p) - \left(\ctwo\right)^2 G(\p) - \cone \ctwo F(\p) \right),
    \end{split}
\end{align}
as a short hand. Similarly, we may use the equation for $\p $ to obtain
\begin{align*}
	I_{2,2} &= \left( \cone F'(\p) + \ctwo G'(\p) \right) \gamma p_{\gamma} \w + S_1,
\end{align*}
where
\begin{align*}
	S_1 = \left(\cone F'(\p) + \ctwo G'(\p)\right)\left|\partialx {p_{\gamma}}\right|^2.
\end{align*}

Now, recall that $I_2 = I_{2,1}+ I_{2,2}$ so that
\begin{equation}
\label{eq:contribution_reaction_terms}
	\begin{split}
	I_2 &=
    \gamma \p \w \left( \cone F'(\p) + \ctwo G'(\p) \right) \\
    &\quad +\px \left( \conex \; F(\p) +  \ctwox \; G(\p) \right) + S_1 + S_2.
    \end{split}
\end{equation}

Recalling the pressure equation, Eq. \eqref{eq:p}, we obtain
\begin{equation}
\label{eq:contribution_pressure_term}
	\begin{split}
	I_1  &= \frac{\partial^2}{\partial x^2}\left(\left|\px\right|^2 + \gamma \p \w\right)\\
    &=2 \fpartial x \left( \px \; \pxx\right) + \gamma \frac{\partial^2}{\partial x^2}(\p \w)\\
    &=2 \fpartial x \left(\px \; \w - \px R\right) + \gamma\frac{\partial^2}{\partial x^2}(\p \w)\\
    &=2 \fpartial x \left(\px \; \w\right) + \gamma \frac{\partial^2}{\partial x^2} (\p \w)  - \underbrace{2\fpartial x\left(\px R\right)}_{I_{1,1}}.
	\end{split}
\end{equation}
Let us note that
\begin{align*}
	I_{1,1} &= \fpartial x \left(\px \, R\right) \\
	&= \frac{\partial^2 \p}{\partial x^2}  \, R + 
    		\px \left( \conex \, F(\p) + \ctwox \, G(\p) \right)  +  S_1\\
    		&=(\w-R)R + 
    		\px \left( \conex \, F(\p) + \ctwox \, G(\p) \right)  +  S_1,
\end{align*}
having used the fact that $\w = \tfrac{\partial^2 \p}{\partial x^2} + R$.

Combining the estimates on the pressure-related term, Eq. \eqref{eq:contribution_pressure_term}, with the reaction-related terms, Eq. \eqref{eq:contribution_reaction_terms}, we obtain
\begin{align}
\label{eq:almost_eqn_for_w}
	\begin{split}
	\partialt \w &= \gamma \left[\frac{\partial^2}{\partial x^2}\left(\p\w\right) + \p\w \left( \cone F'(\p) + \ctwo G'(\p) \right)\right] +2 \fpartial x\left(\px \, \w\right) \\
    &\quad  - 2(\w-R)  R- \px \left( \conex \, F(\p) + \ctwox \, G(\p) \right) - S_1 + S_2.
    \end{split}
\end{align}
We now note that $S_1 \leq 0$ since $F$ and $G$ are decreasing functions, cf. Definition \ref{def:feasible_growth}. Moreover $|S_2| \leq C$, since all the terms in Eq. \eqref{eq:Stwoequation} are uniformly bounded in $\gamma$ and $\e$. Thus we may write $S_2 \geq - |S_2| \geq -C$ whence
\begin{align*}
	\begin{split}
	\partialt \w & \geq \gamma \left[ \frac{\partial^2}{\partial x^2}(\p\w) + \p \w \left( \cone F'(\p) + \ctwo G'(\p) \right) \right] + 2\fpartial x\left(\px \, \w\right) - 2 \w R\\
    &\quad  - \px \left( \conex \, F(\p) + \ctwox \, G(\p) \right) - C.
    \end{split}
\end{align*}
Multiplying by $\sign_{-}(\w)$
yields
\begin{align*}
	\begin{split}
	\fpartial t \neg{\w} &\leq \gamma \left[ \frac{\partial^2}{\partial x^2}\left(\p\neg{\w}\right) + \p\neg{\w} \left( \cone F'(\p)+ \ctwo G'(\p)\right) \right] + 2 \fpartial x\left(\px \neg{\w}\right)  \\
    &\quad - 2\neg{\w} R + \px \left( \conex \, F(\p) + \ctwox G(\p) \right) \mathbbm{1}_{\{\w \leq 0\}} + C,
    \end{split}
\end{align*}
where we used Lemma \ref{lem:krushkov} and Remark \ref{rem:krushkov}. Using the $L^\infty$-bounds on the growth terms  we estimate the right-hand side further
\begin{align}
    \label{eq:101118}
	\begin{split}
	\fpartial t \neg{\w} &\leq \gamma \left[ \frac{\partial^2}{\partial x^2}\left(\p\neg{\w}\right) + \p\neg{\w} \left( \cone F'(\p)+ \ctwo G'(\p)\right) \right]  \\
    &\quad + 2 \fpartial x\left(\px \neg{\w}\right) + C \neg{\w} +  C\left|\px\right| \left( \left|\conex\right| + \left|\ctwox\right| \right)  + C.
    \end{split}
\end{align}
Finally we note that
\begin{align*}
     \cone F'(\p) + \ctwo G'(\p) \leq \max(F'(\p), G'(\p)) \leq -C < 0,
\end{align*}
since $\cone+\ctwo = 1$, which is why Eq.~\eqref{eq:101118} can be further estimated and we obtain
\begin{align*}
	\begin{split}
	\fpartial t \neg{\w} &\leq \gamma \left[ \frac{\partial^2}{\partial x^2}\left(\p\neg{\w}\right) - C \p\neg{\w} \right] + 2 \fpartial x\left(\px \neg{\w}\right) + C \neg{\w} \\
    &\quad  +  C\left|\px\right| \left( \left|\conex\right| + \left|\ctwox\right| \right)  + C.
    \end{split}
\end{align*}
Then, integrating in space yields
\begin{align}
	\label{eq:estimate_negative}
	\ddt \int_{\Omega} \neg{\w}\d x \leq - C \gamma \int_\Omega \p \neg{\w} \d x + C\int_\Omega \neg{\w} \d x+ C  \left\|\px\right\|_{L^\infty(\Omega)}+ C,
\end{align}
where we used the $BV$-estimates on the $\ci$'s~\eqref{eq:estimate_bv}.
Next we show that the $L^\infty$-norm of $\tfrac{\partial \p}{\partial x}$ can be estimated in terms of $|\w|_-$. First, due to Sobolev's embedding theorem in one dimension, we have 
\begin{align*}
	\left\|\px\right\|_{L^\infty(\Omega)} 
	&\leq \left\|\pxx\right\|_{L^1(\Omega )}\\
	&\leq \int_\Omega \left( |\w| + |R| \right) \d x \\
	&=\int_\Omega \left( \w + 2 |\w|_{-} + |R| \right) \d x\\ 
	&\leq\int_\Omega \left( \pxx + 2|R| +2 |\w|_{-} \right) \d x.
\end{align*}
Thus we have
\begin{align}
	\label{eq:2410_2331}
	\left\|\px\right\|_{L^\infty(\Omega)} 
	&\leq C + C\int_\Omega  |\w|_{-} \d x.
\end{align}
Using Eq. \eqref{eq:2410_2331} in Eq. \eqref{eq:estimate_negative} we get
\begin{equation*}
	\ddt \int_{\Omega} \neg{\w} \d x \leq - C \gamma \int_{\Omega} \neg{\w} \d x + C\int_\Omega \neg{\w} \d x + C.
\end{equation*}
The above equation and the Gronwall lemma yield the first estimate of~\eqref{eq:estimate_w}, provided that we can bound $\| \w(0) \|_{L^1(\Omega)}$ independently of $\gamma$, $\e$, which in turn requires a $L^1(\Omega)$ estimate for $\frac{\partial^2 \p}{\partial x^2}(0)$. Such an estimate is provided by~\eqref{eq:technical_w}.

Moreover, recalling \eqref{eq:estimate_negative}, we also get
\begin{align}
\label{eq:estimate_pw_minus}
\gamma \int_0^T \int_\Omega \p \neg{\w} \d x\leq  C.
\end{align}
From Eq. \eqref{eq:p}, we easily infer
\begin{align*}
\gamma \int_{0}^{T}  \int_\Omega  \p \w \d x &= \int_{0}^{T} \int_\Omega \left( \partialt \p - \left|\px\right|^2 \right) \d x \leq \| \p(T) \|_{L^1(\Omega)} -  \| \p(0) \|_{L^1(\Omega)} \leq C.
\end{align*}
The above inequality, together with \eqref{eq:estimate_pw_minus}, completes the proof since it provides the required estimate for $\gamma\, \int_0^T \int_\Omega \p |\w |$. 
\end{proof}

Among byproducts of the previous proof, we highlight the following estimates on $\tfrac{\partial \p}{\partial x}$, $\tfrac{\partial^2 \p}{\partial x^2}$, which will be useful for the proof of the main results.
\begin{corollary}
\label{eq:L1_estimates_px_pxx}
There holds
\begin{align*}
	\left\|\frac{\partial p_{\gamma, \e}}{\partial x}\right\|_{L^\infty(\Omega)} \leq C,\qquad \text{as well as}\qquad 
	\left\|\frac{\partial^2 p_{\gamma, \e}}{\partial x^2}\right\|_{L^1(\Omega)} \leq C.
\end{align*}
\end{corollary}

\section{Proof of the main results}
\label{Section4}
This section is dedicated to passing to the incompressible limit $\gamma \rightarrow +\infty$. As before, we assume that the functions $F_i$, $G_i$ are feasible, that the initial data is well-prepared and that the initial data is compactly supported, \textit{i.e.}, \eqref{InitialSupport}.

\begin{theorem}[Strong compactness of the pressure]
\label{thm_compact_p}
	Let $1\leq q < \infty$ be arbitrary. Then there exists a function $p_\infty\in L^\infty(Q_T)$ such that, upon extraction of a sub-family, there holds
	\begin{align}
		\p \longrightarrow p_\infty,
	\end{align}
	pointwise and strongly in $L^q(Q_T)$.
\end{theorem}

\begin{proof}
For a given sequence $(u_\gamma)_{\gamma}$ defined on $(0,T)\times\Omega$ and bounded in $L^1(Q_T)$, we recall that if we control both the time shifts and space shifts as follows
\begin{align}
	\int_0^{T-h} \int_\Omega \left| u_\gamma(t+h,x+a) - u_\gamma(t,x) \right| \rightarrow 0,
\end{align}
as $(h,a) \rightarrow (0,0)$, independently of $\gamma$,
then $(u_\gamma)_{\gamma}$ has compact closure in $L^1(Q_T)$ by the Fr\'echet-Kolmogorov compactness theorem. This is of course true if the following stronger estimate holds:
\begin{align}
 	\int_0^T \int_\Omega \left( \left|\partial_t u_\gamma \right| + \left|\partial_x u_\gamma\right| \right)\d x\d t \leq C.
\end{align}
Finally, if furthermore $(u_\gamma)_\gamma$ is (uniformly in $\gamma$) in $L^\infty(Q_T)$, then it is also compact in $L^q(Q_T)$ for any $1 \leq q < \infty$, after applying Lebesgue's dominated convergence theorem.
This general result is crucial in passing to the limit $\gamma \rightarrow +\infty$ and will use it in this proof, abusively referring to it as the Fr\'echet-Kolmogorov Theorem.
From the estimates of Theorem~\ref{thm_apriori}, we clearly have
\begin{align*}
\int_0^T \int_\Omega \left|\partialt \p\right|  \d x\d t\leq \int_0^T \int_\Omega \left( \left|\px\right|^2 + \gamma \p |\w| \right)\d x \d t \leq C.
\end{align*}
Thanks to the compactness assumption, we have the required bound:
\begin{align}\label{eq:compactness_p}
	\int_0^T \int_\Omega \left(\left|\partial_t \p\right| + \left|\partial_x \p\right|\right)\d x\d t \leq C, 
\end{align}
from which strong convergence in $L^1(Q_T)$ follows.
Note that this convergence holds even pointwise after possibly passing to another sub-sequence. Finally, since $\p \leq P_H$, this bound also holds at the limit.
\end{proof}

\begin{theorem}[Complementarity formula]
\label{thm_complementarity}
We may pass to the limit in Eq. \eqref{eq:p} to obtain the complementarity relation 
\begin{align}
	p_\infty \left[\frac{\partial^2 p_\infty}{\partial x^2}  + n_\infty^{(1)}F(p_\infty) + n_\infty^{(2)}G(p_\infty) \right] = 0,
\end{align}
in $\mathcal{D}'(Q_T)$, where $n_\infty^{(i)}$, $i=1,2$ and $p_{\infty}$  weakly satisfy the equations
\begin{equation}
\label{eq:limit_np}
\begin{cases}
\partialt {n_{\infty}^{(1)}} = & \fpartial x \left(n_{\infty}^{(1)} \partialx {p_{\infty}} \right) + n_{\infty}^{(1)} F_1(p_{\infty}) + n_{\infty}^{(2)} G_1(p_{\infty}),\\
\partialt {n_{\infty}^{(2)}} = & \fpartial x \left(n_{\infty}^{(2)} \partialx {p_{\infty}} \right) + n_{\infty}^{(1)} F_2(p_{\infty}) + n_{\infty}^{(2)} G_2(p_{\infty}),
\\
\qquad 0= &p_{\infty} (1-n_{\infty}) = 0 \qquad \text{a.e.},
\end{cases}
\end{equation}
starting from $n_{\infty}^{(i)}(0) = n_{\infty, \rm init}^{(i)}$, $i=1,2$, where $n_{\infty} = n_{\infty}^{(1)} + n_{\infty}^{(2)}$.
\end{theorem}

Before we begin the proof of the complementarity formula in the incompressible limit, we recall some properties of mollifiers and convolutions.
\begin{remark}[\label{rem:DiracNMoll}Properties of mollifiers]
We set
	\begin{align*}
	\varphi(x):= \left\{ 
	\begin{array}{ll}
		0,&|x|\geq1,\\
		M\exp \left(-\frac{1}{1-|x|^2}\right),& |x|<1,
		\end{array}
		\right.
	\end{align*}
	and recall that $\varphi\in C_c^\infty(\R)$ is a nonnegative, symmetric function. We choose $M$ so that $\varphi$ has mass one. Furthermore, we define the Dirac sequence $(\varphi_\delta)_{\delta> 0}$ by
\begin{equation*}
\varphi_\delta(x) = \delta^{-1} \varphi\left(\delta^{-1}x\right).
\end{equation*}
Then,  there holds, for any function $f\in W^{1,q}(\Omega)$ (with $K:= \int_\mathbb{R} |x| \varphi(x) \, dx$)
\begin{equation*}
	\|f\star\varphi_\delta - f\|_{L^q(\Omega)} \leq K \|f'\|_{L^q(\Omega)} \delta .
\end{equation*}
Moreover the derivative is given by
\begin{align*}
	\varphi_\delta'(x) = -2 \frac{\delta ^{-1}x}{\left[1 - (\delta^{-1}x)^2\right]^2} \varphi_\delta(x),
\end{align*}
whence $\|\varphi_\delta(x)\|_{L^1(\Omega)}\leq C\delta^{-1}$.
\end{remark}
We have now gathered all information necessary for passing to the limit in the pressure equation, Eq. \eqref{eq:p}.

\begin{proof}
We rewrite the equation for $\p$ and multiply by a test function in order to obtain 
\begin{equation*}
	\frac{1}{\gamma}\int_0^T\int_\Omega \phi(t,x) \left( \partialt \p - \left|\px\right|^2\right) \d x\d t = \int_0^T \int_\Omega \phi(t,x) \p \w \d x\d t.
\end{equation*}
From the bounds
\begin{equation*}
\int_0^T \int_\Omega \left|\partialt \p\right| \d x\d t \leq C, \; \int_0^T \int_\Omega  \left|\partialx \p\right|^2\d x\d t \leq C,
\end{equation*}
the left-hand side must converge to zero, meaning
that \begin{align}
		\p \w \longrightarrow 0,
	\end{align}
in the distributional sense. It now remains to identify the limit. We write 
\begin{equation*}
    \int_0^T \int_\Omega \phi \p \w  \d x= \int_0^T \int_\Omega \phi  \p \pxx \d x + \int_0^T \int_\Omega \phi \p  \left(\cone F(\p) +  \ctwo G(\p)\right)\d x,
\end{equation*}
and we treat both terms independently. 
\par

\textit{First term.}
For the first one, we write 
\begin{equation*}
	\p \frac{\partial^2 \p}{\partial x^2} = \frac{1}{2} \frac{\partial^2\p^2}{\partial x^2} - \left(\px\right)^2,
\end{equation*}
leading to
\begin{align*}
	\int_0^T \int_\Omega \phi  \p \pxx \d x = \frac{1}{2} \int_0^T\int_\Omega \frac{\partial^2 \phi}{\partial x^2} \,  \p^2 \d x - \int_0^T\int_\Omega \phi \left|\px\right|^2\d x.
\end{align*}
We may pass to the limit in the first term by Theorem \ref{thm_compact_p}. The second term requires to analyse compactness for $\tfrac{\partial \p}{\partial x}$, a problem we again approach with the Fr\'echet-Kolmogorov Theorem as the main tool. 
Its space derivative is already controlled since, from Corollary~\ref{eq:L1_estimates_px_pxx}, 
\begin{align*}
	\sup_{0 \leq t \leq T}\left\|\pxx(t)\right\|_{L^1(\Omega)} \leq C. 
\end{align*}

For the time derivative we will use the Fr\'echet-Kolmogorov  compactness method, we shall prove that, as $h$ and tends to $0$,
\begin{align*}
\int_0^{T-h} \int_\Omega \left|\px(t+h,x) - \px (t,x) \right| \d x\d t \rightarrow 0.
\end{align*}

Let us continue with the analysis. For the ease of notations, we set
\begin{equation*}
	u_h(t,x):=  \px(t+h,x) - \px (t,x).
\end{equation*}
By comparing $u$ to its mollified version, the triangular equality yields
\begin{align}
\label{eq:timeshift_estimate}
\begin{split}
\int_0^{T-h} \int_\Omega |u_h(t,x)| \, \d x \d t &\leq \int_0^{T-h}\int_\Omega |u_h(t,x) - u_h(t,\cdot)*\varphi_\delta(x)| \, \d x \d t\\
&\qquad   + \int_0^{T-h}\int_\Omega |u_h(t,\cdot)*\varphi_\delta(x)| \, \d x \d t. 
\end{split}
\end{align}

Here, $\delta$ is a function (to be specified later on) of $h$ tending to $0$. 
By Remark \ref{rem:DiracNMoll}, there holds 
\begin{equation*}
\int_0^{T-h}\int_\Omega |u_h(t,x) - u_h(t,\cdot)*\varphi_\delta(x)| \, \d x \d t \leq C h \int_0^{T-h}  \left|\left|\pxx(t+h) - \pxx (t)\right|\right|_{L^1(\Omega)} \d t \leq C h,
\end{equation*}
thanks to Corollary \ref{eq:L1_estimates_px_pxx}.
which proves that the right-hand side converges to zero as $\delta \to 0$, uniformly in $\gamma$. It suffices to show that the same result holds for the second integral.
We write
\begin{align*}
\int_0^{T-h}\int_\Omega \left|u_h(t,\cdot)*\varphi_\delta(x)\right| \, \d x\d t
& \leq \int_0^{T-h}\int_\Omega |(\p(t+h) - \p(t)) \star \varphi_\delta'(x)|d x\d t,
\end{align*}
after exchanging derivatives in the convolution. We now bound with the estimate on the derivative of the mollifier, cf. Remark \ref{rem:DiracNMoll}:
\begin{align*}
\int_0^{T-h}\int_\Omega |u_h(t,\cdot) \star \varphi_\delta| \d x \d t \leq & \int_0^{T-h}\int_\Omega |\varphi'_\delta(y)| \d y \, \int_\Omega|\p(t+h,x) -\p(t,x)| \,\d x \d t \\
\leq  &\frac{C}{\delta} \int_0^{T-h}\int_\Omega |\p(t+h,x)- \p(t,x)| \,\d x\d t \\
\leq & \frac{C}{\delta} \int_0^{T-h}\int_\Omega \abs{ \int_t^{t+h} \partialt \p(s,x) \d s} \d x \d t.
\end{align*}
We rearrange the integrals and obtain
\begin{align*}
\int_0^{T-h}\int_\Omega |u_h(t,\cdot) \star \varphi_\delta| \d x \d t
\leq & \frac{C}{\delta} \int_\Omega \int_0^{T-h} \int_t^{t+h} \left|\partialt \p(s,x)\right| \d s \d x \d t\\
=& \frac{C}{\delta} \int_\Omega \int_{s=0}^T\int_{t=\max(0,s-h)}^{\min(T-h,s)} \left|\partialt \p (s,x)\right|\d t\, \d s\,  \d x \\
=& \frac{C}{\delta}\int_\Omega \int_0^T (\min(T-h,s) - \max(0,s))\left|\partialt \p \right|\d s \d x\\
\leq & \frac{Ch}{\delta}\int_\Omega\int_0^T\left|\partialt \p(s,x)\right| \d s\, \d x.
\end{align*}
Using the fact that $\tfrac{\partial \p}{\partial t}$ is bounded in $L^1(0,T;L^1(\Omega))$ uniformly in $\gamma$, cf. \eqref{eq:compactness_p}, we obtain
\begin{align*}
\int_0^{T-h}\int_\Omega |u_h(t,\cdot) \star \varphi_\delta| \d x \d t \leq C \sqrt{h},
\end{align*}
having set $\delta = \sqrt{h}$.
Thus we conclude that the entire right-hand side of Eq. \eqref{eq:timeshift_estimate} converge to  zero as $h\to0$. Thus the time shifts are also controlled and we may infer the strong compactness of $\tfrac{\partial \p}{\partial x}$.\\

\textit{Second term.}
For the second term involving $\p  (\cone F(\p) + \ctwo G(\p))$, we note that
\begin{align*}
	\p  (\cone F(\p) + \ctwo G(\p)) = \p^{1 - 1/\gamma}  (\n^{(1)} F(\p) + \n^{(2)} G(\p)).
\end{align*}
Passing to the limit requires weak convergence of $\none$ and $\ntwo$, since the strong convergence of the pressure would then allow us to pass to the limit in the second term, \textit{i.e.},
\begin{align*}
	\int_0^T \int_\Omega \phi \p  \left(\cone F(\p) +  \ctwo G(\p)\right)\d x \rightarrow 	\int_0^T \int_\Omega \phi p_\infty  \left(n_\infty^{(1)} F(p_\infty) +  n_\infty^{(2)} G(p_\infty)\right)\d x.
\end{align*}

For the convergence of $\none$, $\ntwo$, we use the Fr\'echet-Kolmogorov Theorem. For the space derivative, the result is already provided by estimate~\eqref{eq:estimate_bv}.
For the time derivative, it suffices to use the equation for the $n_\gamma^{(i)}$'s. We focus on $\none$ and expand the divergence term to get 
\begin{equation*}
\partialt {\none} = \partialx \none \partialx \p +  \none\, \pxx + \none F_1(\p) + \ntwo G_1(\p).
\end{equation*}
The two last terms are in $L^\infty(\Omega)$ while the two first terms are controlled in $L^1(\Omega)$ thanks 
to Corollary \ref{eq:L1_estimates_px_pxx}. Consequently, we have strong convergence of the densities $(\none, \ntwo)$ to some $(n_\infty^{(1)}, n_\infty^{(2)})$ in $L^1(Q_T)$. 

\textit{Limit equation for $n_\infty^{(1)}$, $n_\infty^{(2)}$.} 
We aim at passing to the limit in 
\begin{equation*}
\partialt {\n^{(1)}} = \fpartial x \left(\n^{(1)} \partialx {\p} \right) + \none F_1(\p) + \ntwo G_1(\p).
\end{equation*}

The reaction terms readily pass to the limit since $\p$ converges strongly and the $n_\gamma^{(i)}$'s also. For the divergence term, we use the same results and the strong convergence of $\tfrac{\partial \p}{\partial x}$ established above.

\textit{Initial condition.} 
The limit Cauchy problem is completely identified with the initial condition $n_{\infty, \rm init}^{(1)}$, $n_{\infty, \rm init}^{(2)}$, thanks to~\eqref{eq:init}.

\textit{Other relations.}
Equation~\eqref{eq:limit_np} is obtained by writing $\n \p^{\frac{\gamma}{\gamma-1}} = \p$, and using the convergences of $\p$ and $\n$, respectively, to conclude.

Note that similar arguments allow to prove the strong convergence of $\cone$, $\ctwo$ to some $c^{(1)}_\infty$, $c^{(2)}_\infty$. These limits will satisfy the relations 
\begin{equation*}
c^{(i)}_\infty n_\infty = n^{(i)}_\infty, \qquad c_\infty^{(i)} p_\infty  = n_\infty^{(i)} p_\infty, \qquad i=1,2.
\end{equation*}
\end{proof}

\begin{remark}
We again emphasise that the solutions of the regularised system converge to those of the original one as $\e$ tends to $0$. This proves that the limit system obtained by letting both $\e$ tend to $0$ and $\gamma$ tend to $+\infty$ is also the system obtained from the original one by letting $\gamma$ tend to $+\infty$. 
\end{remark}

\section{Segregation property}
\label{Section5}

In order to establish the preservation of the segregation property $n^{(1)}n^{(2)} \equiv 0$, we begin with extracting subsequences such that the initial population fractions pass to the limit $\gamma \rightarrow \infty$, $\e \rightarrow 0$. More precisely, we notice that upon extracting subsequences,
there exist $c^{(1)}_{\infty, \rm init}$, $c^{(2)}_{\infty, \rm init} \in L^1(\Omega_0)$ such that, as $\e \rightarrow 0$, $\gamma \rightarrow \infty$, 
\begin{equation}
\label{eq:initial_fraction}
    c^{(i)}_{\gamma, \e}(0) \rightarrow c^{(i)}_{\infty, \rm init} \qquad \text{in } L^1(\Omega_0). 
\end{equation}
Indeed, Helly's selection theorem applies since the quotients are bounded both in $L^\infty$ and in BV by assumption~\eqref{eq:BVcond_init}. 

Our approach is then based on the observation that, in the absence of cross-reactions ($F_2 = G_1 = 0$),  a  direct manipulation shows that $\n \, \cone \ctwo$ satisfies the equation 
\begin{equation}
    \label{egregationFund}    
\partialt{ \n \, \cone \ctwo}  =\fpartial x\left( \n \, \cone \ctwo\px\right) + \n \, \cone \ctwo\left(\ctwo F_1(\p) + \cone G_2(\p) \right).
\end{equation}
If initially $ \n \, \cone \ctwo =0$, any weak solution will propagate the segregation property. We explain below that we can pass to the limit in this equation in the weak sense.

\begin{theorem} [Segregation property for $n^{(1)}_{\infty} , \, n^{(2)}_{\infty} $]
Under the assumptions of Theorem~\ref{thm_complementarity} and in the absence of cross-reactions ($F_2 = G_1 = 0$), equation~\eqref{egregationFund}  holds for the limits when $\gamma \to \infty$. The segregation property holds, \textit{i.e.}, with the notations of Theorem~\ref{thm_complementarity}, if $n^{(1)}_{\infty, \rm init}  \,\, n^{(2)}_{\infty, \rm init} = 0$, a.e. in $\Omega_0$, then $n^{(1)}_{\infty} (t)\, n^{(2)}_{\infty}(t) = 0$  a.e. in $Q_T$.
\end{theorem}

\begin{proof} Firstly, we pass to the limit in Equation~\eqref{egregationFund}. Then, writing $ \n(0, \cdot) \cone(0, \cdot) \ctwo(0, \cdot) =  \none(0, \cdot) \ctwo(0, \cdot)$, 
we can pass to the limit in the intial data since both terms are in $L^\infty(\Omega_0)$ and have a limit in $L^1(\Omega_0)$.

Applying the convergence results in Theorem~\ref{Complementarity}, we can also pass to the limit in the weak form of Equation~\eqref{egregationFund}, see~\eqref{eq:weak}. Therefore we obtain for $u= n_{\infty}c_{\infty}^{(1)}c_{\infty}^{(2)}$, the equation
\begin{equation}
 \label{segregation:equ}
	\begin{cases}
	\partialt {u} = \fpartial x \left(u \partialx {p_{\infty}} \right) +  u \left( c^{(2)}_\infty F_1(p_{\infty}) +  c^{(1)}_\infty G_2(p_{\infty}) \right), 
	\\
	u(0,\cdot)= n_{\infty, \rm init} c^{(1)}_{\infty, \rm init}  c^{(2)}_{\infty, \rm init} \in L^1(\Omega_0)  .
	\end{cases}
\end{equation}	 	
	
Secondly, we wish to show that $u(t,\cdot)=0$ when $u(0,\cdot)=0$. To do so, we use the definition of the weak form~\eqref{eq:weak} with a compactly supported  test function $ \phi(t,x)$ which takes the value $1$ on $Q_T$ and arrive to
\begin{equation*}
\ddt \int u(t,x) \phi(t,x) \d x \leq \| c^{(2)}_\infty F_1(p_{\infty}) +  c^{(1)}_\infty G_2(p_{\infty}) \|_\infty \int u(t,x) \phi(t,x) \d x .
\end{equation*}
Using Gronwall's lemma, we conclude that $ \int u(t,x) \phi(t,x) \d x  = \int_{Q_T}  u(t,x) \d x=0$ and thus $u\equiv 0$.
\end{proof}

\begin{remark} This proof of the segregation property can be adapted to more general solutions of~\eqref{segregation:equ} than those with compact support. Using test functions with a truncation parameter, we merely need $u\geq 0$ be integrable and $\tfrac{\partial p_{\infty}}{\partial x} $ be bounded. 
\end{remark}

\begin{remark}
   The product $v:= \none \ntwo$, is not well adapted because it solves 
   \begin{align*}    \partialt v =\fpartial x\left(v \px\right) + v\left(F_1(\p) + G_2(\p) + \pxx\right).
\end{align*}
This is not a well defined equation  in the limit,  with the available regularity on $p_\infty$. Indeed, $ \tfrac{\partial^2 \p}{\partial x^2} $ is bounded in $L^1$ but  $\tfrac{\partial^2 p_\infty}{\partial x^2}$ is  just a measure at the limit~\cite{Perthame2014b}.
	In other words, one cannot hope to derive the segregation property from the above equation on $v$.
\end{remark}

\section{Conclusions and open questions}
\label{Section6}

We have established the incompressible limit of the two-species system~\eqref {System} in one space dimension. The mathematical interest arises from vacuum states which generate a free boundary described by a Hele-Shaw type system. Our approach is based on an extension of the Aronson-B\'enilan estimates which we use in an $L^1$ setting rather than using upper bounds as usually done. Any improvement in the method and the estimate itself could be of interest. There are three major difficulties to extend this estimate to higher dimension. Firstly, we work in BV as in~\cite{Carrillo2017}. Secondly, some exchanges of derivatives cannot be performed in more than one dimension. Thirdly, we use that $\p$ is Lipschitz continuous (using Sobolev injections) which is crucial in estimates such  as~\eqref{eq:estimate_negative}.
\\

Several extensions could be of interest but require new ideas. The question of the regularity theory for the free boundary is completely open and faces the difficulty of weak estimates compared to the one species case in~\cite{MePeQu}.  Including drift terms is of interest in view of \cite{Kim_Yao, DaMeMaRo, MeSa20162}.
Also including different mobilities for the two species  as in~\cite {LoLoPe} is an open question. 
\medskip

\textbf{Acknowledgments.} F.B. and B.P. have received funding from the European Research Council (ERC) under the European Union's Horizon 2020 research and innovation program (grant agreement No 740623). M.S. acknowledges the kind invitation to LJLL funded by the previous grant. Furthermore, M.S. received funding for two research visits from the Doris Chen Mobility Award awarded by Imperial College London. C.P.
acknowledges support from the Swedish Foundation of Strategic Research grant AM13-004.

\appendix

\section{Variations of Krushkov}
Here we present the rigorous argument for passing the negative part and the absolute value into derivatives in a conservation law 
\begin{align}
	\label{eq:convectiondiffusion}
		\partialt f = \fpartial x\left(af\right) + \frac{\partial^2}{\partial x^2}(bf)
\end{align}
with no-flux boundary conditions.
\begin{lemma}
\label{lem:krushkov}
	Assume  $a\in C_b^1(\R,\R)$ and $b\in C_b^2(\R, \R_+)$. Let $\phi \in C^1(\R,  \R_+)$ be a twice differentiable function, and assume that $f$ satisfies Eq. \eqref{eq:convectiondiffusion}.
	Then there holds
	\begin{align*}
				\fpartial t \phi(f) = \fpartial x\left(a\phi(f)\right) + \frac{\partial^2}{\partial x^2}(b\phi(f)) + \left(\partialx a + \frac{\partial^2 b}{\partial x^2}\right)\left[f\phi'(f) - \phi(f)\right] - a \phi''(f) \left|\partialx f\right|^2.
	\end{align*}
	Moreover, if $(\phi_\e)_{\e > 0}$ is a family of smooth and convex approximations of $\phi\in\{|\cdot|_-,|\cdot|_+,|\cdot|\}$ then there holds for almost every $x\in \Omega$ and $t>0$
	\begin{align*}
				\fpartial t \phi(f) \leq  \fpartial x\left(a\phi(f)\right) + \frac{\partial^2}{\partial x^2}(b\phi(f)) .
	\end{align*}
\end{lemma}
\begin{proof}
	Multiplying Eq. \eqref{eq:convectiondiffusion} by $\phi_\varepsilon'(f)$ yields the first statement after reversing some applications of the chain rules and algebraic simplifications. As for the second statement, we note that $-a \phi_\varepsilon''(f)| \partial_x f|^2  \leq 0$,	independently of $\varepsilon>0$. Moreover, due to the boundedness of $\partial_x a + \partial_x^2 b$ we obtain
	\begin{align*}
				\fpartial t \phi_\varepsilon(f) \leq \fpartial x\left(a\phi_\varepsilon(f)\right) + \frac{\partial^2}{\partial x^2}(b\phi_\varepsilon(f)) + C \left\|f\phi_\varepsilon'(f) - \phi_\varepsilon(f)\right\|_{\infty}.
	\end{align*}
	Due to the uniformity of the approximation, the norm is of order $\varepsilon$ and passing to the limit yields the second statement.
\end{proof}

\begin{lemma}
Assume  $a\in C_b^1(\R,\R)$ and $b\in C_b^2(\R, \R_+)$. Let $\phi \in C^1(\R,  \R_+)$ be a convex function, then there holds
	\begin{align*}
		\ddt \int_\Omega \phi(f) \d x &\leq C \int_\Omega \left[\phi'(f)f -\phi(f)\right] \left(\partialx a + \frac{\partial^2 b}{\partial x^2}\right)\d x.
	\end{align*}
	Moreover, if $(\phi_\e)_{\e > 0}$ is a family of smooth and convex approximations of $\phi\in\{|\cdot|_-,|\cdot|_+,|\cdot|\}$ then there holds for almost every $x\in \Omega$ and $t>0$
	\begin{align*}
		0 \leq \phi(f(t)) \leq \phi(f(0)) .
	\end{align*}
	
\end{lemma}
\begin{proof}
	The proof is straightforward and we give it here only for the reader's convenience. We compute
	\begin{align*}
		\ddt \int_\Omega \phi(f)\d x =\underbrace{\int_\Omega \phi'(f) \fpartial x\left(af\right)\d x }_{I_1} + \underbrace{\int_\Omega  \phi'(f) \frac{\partial^2}{\partial x^2}(bf)\d x }_{I_2}.
	\end{align*}
	We treat the advective term, $I_1$, and the diffusion term, $I_2$, separately. We have
	\begin{align*}
		I_1 = \int_\Omega \partialx a \phi'(f)f + a\frac{\partial \phi(f)}{\partial x}\d x =\int_\Omega \partialx a \left[\phi'(f)f-\phi(f)\right]\d x + \int_{\partial \Omega}a \phi(f)\d S_x.
	\end{align*}
	As for the second term we obtain
	\begin{align*}
		I_2 &= \int_\Omega \phi'(f)f \frac{\partial^2 b}{\partial x^2} + 2\partialx b \partialx {\phi(f)} + \phi'(f) b\frac{\partial^2 f}{\partial x^2}\d x\\
		&=\int_\Omega \phi'(f)f \frac{\partial^2 b}{\partial x^2} + \partialx b \partialx {\phi(f)}+ \partialx b \partialx {\phi(f)} + \phi'(f) b\frac{\partial^2 f}{\partial x^2}\d x\\
		&=\int_\Omega \left[\phi'(f)f - \phi(f)\right] \frac{\partial^2 b}{\partial x^2} + \partialx b \partialx {\phi(f)} + \phi'(f) b\frac{\partial^2 f}{\partial x^2} \d x + \int_{\partial \Omega}\partialx b \phi(f)\d S_x,
	\end{align*}
	by an integration by parts in the second term. Another integration by parts then yields
	\begin{align*}
	I_2 	&=\int_\Omega \left[\phi'(f)f - \phi(f)\right] \frac{\partial^2 b}{\partial x^2}  - b \phi''(f) \left|\frac{\partial f}{\partial x}\right|^2 \d x + \int_{\partial \Omega}\partialx b \phi(f) +  b \partialx {\phi(f)} \, \d S_x.
		\end{align*}
	Combining $I_1$ and $I_2$ we obtain
	\begin{align*}
		\ddt \int_\Omega \phi(f) \d x &= \int_\Omega \left[\phi'(f)f -\phi(f)\right] \left(\partialx a + \frac{\partial^2 b}{\partial x^2}\right)\d x + \int_\Omega b \phi''(f)\left|\partialx f\right|^2\d x \\
		&\quad + \int_{\partial \Omega} a \phi(f)  + \partialx b \phi(f) + b \frac{\partial \phi(f)}{\partial x} \d S_x.
	\end{align*}
	Using the fact that $b\geq 0$ and the fact that $a,b$ have bounded derivatives we get
	\begin{align*}
		\ddt \int_\Omega \phi(f) \d x &\leq C \left|\phi'(f)f -\phi(f)\right| + \int_{\partial \Omega} a \phi(f)  +\fpartial x\left(b \phi(f)\right) \d S_x.
	\end{align*}
	Finally, using the no flux condition we obtain
	\begin{align}
		\label{eq:2510_1324}
		\ddt \int_\Omega \phi(f) \d x &\leq C \left|\phi'(f)f -\phi(f)\right|,
	\end{align}
	which concludes the first part of the proof. For the second statement we simply note that the inequality is satisfied for any $\e>0$. It is classical that the modulus, the positive part, and the negative part can be uniformly approximated such that the right-hand side of Eq. \eqref{eq:2510_1324} is of order $\e$. Passing to the limit we obtain
	\begin{align*}
		\ddt \int_\Omega \phi(f)\d x \leq 0,
	\end{align*}
	whence $0 \leq \phi(f(t)) \leq \phi(f(0))$, which concludes the proof.
\end{proof}
\begin{remark}
	\label{rem:krushkov}
	Note that the assumptions on the functions $a,b$ can be weakened from having `bounded derivatives' to having integrable derivatives, \textit{i.e.},  $\partial_x a, \partial_{xx} b \in L^q(\Omega,\R)$.
\end{remark}

\section{Energy}
\label{Energy}

\begin{proposition}
Let $H_1(p) := \int_0^{p} F(z) \, \d z$ and $H_2(p) := \int_0^{p}  G(z) \, \d z$ for $p \geq 0$. Then, the energy
\begin{equation*}
\mathcal{E}(t) := \int_\Omega \left(  \frac{1}{2}\left|\frac{\partial p_{\gamma, \e}}{\partial x} \right|^2 - c^{(1)}_{\gamma,\e} H_1(p_{\gamma, \e}) - c^{(2)}_{\gamma,\e} H_2(p_{\gamma, \e})\right) \d x
\end{equation*}
is such that, for a constant  $C$ is independent of $\gamma$ and $\e$,
\begin{equation}
\label{eq:energy_inequality}
\mathcal{E}'(t) + \gamma \int_{\Omega} p_{\gamma,\e} w_{\gamma,\e} ^2 \d x \leq C.
\end{equation}

\end{proposition}
\begin{proof}
Consider the equation for the pressure \eqref{eq:p} and multiply by $- \tfrac{\partial^2 \p}{\partial x^2}$. Integration by parts yields
\begin{align}
\label{eq:laplacian_pressure}
 \frac{1}{2}\ddt \int_\Omega  \left|\px\right|^2 \d x +\gamma \int_\Omega \p \left|\pxx\right|^2 \d x + \gamma\int_\Omega \p\pxx R \, \d x = 0.
\end{align}
Moreover, using the equations for $\cone$ and $\ctwo$ we compute
\begin{equation}
\label{eq:c1H1}
\begin{split}
\frac{\partial \left( \cone H_1(\p) \right)}{\partial t} =& H_1(\p) \left( \conex \px + \cone F_1(\p) + \ctwo G_1(\p) - (\cone)^2F(\p) - \cone\,\ctwo G(\p)\right) \\ &+ \cone F(\p) \left[ \left|\px\right|^2 + \gamma \p \w \right],
\end{split}
\end{equation}
and
\begin{equation}
\label{eq:c2H2}
\begin{split}
\frac{\partial \left( \ctwo H_2(\p) \right)}{\partial t} &= H_2(\p) \left( \ctwox \px + \cone F_2(\p) + \ctwo G_2(\p) - (\ctwo)^2\,G(\p) - \cone \, \ctwo F(\p)\right)\\& + \ctwo G(\p) \left[\left|\px\right|^2 + \gamma \p \w \right].
\end{split}
\end{equation}

Summing \eqref{eq:laplacian_pressure}, \eqref{eq:c1H1} and \eqref{eq:c2H2}, and using the uniform bounds for $\cone$, $\ctwo$ and the reaction terms, we get
\begin{align*}
\ddt \int_\Omega \bigg( \frac12 \left| \px \right| ^2 -& \cone H_1(\p) - \ctwo H_2(\p) \bigg) \d x + \gamma \int_{\Omega} \p \w ^2 \d x \leq \\ & \quad C \int_{\Omega} \left[ \left|\px \right|^2 + \left| \px \right| \left( \left|\conex\right| + \left|\ctwox\right|\right) \right] \d x.
\end{align*}
Theorem \ref{thm_apriori}, together with the H\"older inequality and the Sobolev embeddings yield the desired bound \eqref{eq:energy_inequality}.
\end{proof}

\bibliography{Hele_Shaw}

\begin{thebibliography}{10}

\bibitem{Aronson1979}
Donald~G Aronson and Philippe B{\'e}nilan.
\newblock R{\'e}gularit{\'e} des solutions de l'{\'e}quation des milieux poreux
  dans $\mathbb{R}^n$.
\newblock {\em CR Acad. Sci. Paris S{\'e}r. AB}, 288(2):A103--A105, 1979.

\bibitem{BDPM10}
M.~Bertsch, R.~Dal~Passo, and M.~Mimura.
\newblock A free boundary problem arising in a simplified tumour growth model
  of contact inhibition.
\newblock {\em Interfaces and Free Boundaries}, 12(2):235--250, 2010.

\bibitem{BGH87a}
M~Bertsch, ME~Gurtin, and D~Hilhorst.
\newblock On interacting populations that disperse to avoid crowding: the case
  of equal dispersal velocities.
\newblock {\em Nonlinear Analysis: Theory, Methods \& Applications},
  11(4):493--499, 1987.

\bibitem{BHIM12}
Michiel Bertsch, Danielle Hilhorst, Hirofumi Izuhara, and Masayasu Mimura.
\newblock A nonlinear parabolic-hyperbolic system for contact inhibition of
  cell-growth.
\newblock {\em Differ. Equ. Appl}, 4(1):137--157, 2012.

\bibitem{BCGR}
Didier Bresch, Thierry Colin, Emmanuel Grenier, Benjamin Ribba, and Olivier
  Saut.
\newblock Computational modeling of solid tumor growth: the avascular stage.
\newblock {\em SIAM J. Sci. Comput.}, 32(4):2321--2344, 2010.

\bibitem{BT83}
Stavros~N. Busenberg and Curtis~C. Travis.
\newblock Epidemic models with spatial spread due to population migration.
\newblock {\em Journal of Mathematical Biology}, 16(2):181--198, Jan 1983.

\bibitem{ByDr}
H.~M. Byrne and D.~Drasdo.
\newblock Individual-based and continuum models of growing cell populations: a
  comparison.
\newblock {\em Math. Med. Biol.}, 58(4-5):657--687, 2003.

\bibitem{Carrillo2017}
J.~A. Carrillo, S.~Fagioli, F.~Santambrogio, and M.~Schmidtchen.
\newblock Splitting schemes and segregation in reaction cross-diffusion
  systems.
\newblock {\em SIAM J. Math. Anal.}, 50(5):5695--5718, 2018.

\bibitem{Kim_Yao}
Katy Craig, Inwon Kim, and Yao Yao.
\newblock Congested aggregation via {N}ewtonian interaction.
\newblock {\em Arch. Ration. Mech. Anal.}, 227(1):1--67, 2018.

\bibitem{DaMeMaRo}
Julien Dambrine, Nicolas Meunier, Bertrand Maury, and Aude Roudneff-Chupin.
\newblock A congestion model for cell migration.
\newblock {\em Commun. Pure Appl. Anal.}, 11(1):243--260, 2012.

\bibitem{Degond2018}
Pierre Degond, Sophie Hecht, and Nicolas Vauchelet.
\newblock Incompressible limit of a continuum model of tissue growth for two
  cell populations.
\newblock {\em arXiv preprint arXiv:1809.05442}, 2018.

\bibitem{GP84}
Morton~E Gurtin and AC~Pipkin.
\newblock A note on interacting populations that disperse to avoid crowding.
\newblock {\em Quarterly of Applied Mathematics}, pages 87--94, 1984.

\bibitem{Gwiazda2018}
Piotr Gwiazda, Beno{\^\i}t Perthame, and Agnieszka {\'S}wierczewska-Gwiazda.
\newblock A two species hyperbolic-parabolic model of tissue growth.
\newblock {\em arXiv preprint arXiv:1809.01867}, 2018.

\bibitem{Hecht2017}
Sophie Hecht and Nicolas Vauchelet.
\newblock Incompressible limit of a mechanical model for tissue growth with
  non-overlapping constraint.
\newblock {\em Communications in mathematical sciences}, 15(7):1913, 2017.

\bibitem{Kpo}
Inwon Kim and Norbert Po{\v z\'a}r.
\newblock Porous medium equation to {H}ele-{S}haw flow with general initial
  density.
\newblock {\em Trans. Amer. Math. Soc.}, 370(2):873--909, 2018.

\bibitem{KTu}
Inwon Kim and Olga Turanova.
\newblock Uniform convergence for the incompressible limit of a tumor growth
  model.
\newblock {\em Ann. Inst. H. Poincar\'e Anal. Non Lin\'eaire},
  35(5):1321--1354, 2018.

\bibitem{KPS}
Inwon~C. Kim, Beno{\^ \i}t Perthame, and Panagiotis~E. Souganidis.
\newblock Free boundary problems for tumor growth: a viscosity solutions
  approach.
\newblock {\em Nonlinear Anal.}, 138:207--228, 2016.

\bibitem{LoLoPe}
Tommaso Lorenzi, Alexander Lorz, and Beno\^{i}t Perthame.
\newblock On interfaces between cell populations with different mobilities.
\newblock {\em Kinet. Relat. Models}, 10(1):299--311, 2017.

\bibitem{LowengrubFrieboes_etal2010}
J.~S. Lowengrub, H.~B. Frieboes, F.~Jin, Y.-L. Chuang, X~Li, P.~Macklin, S.~M.
  Wise, and V.~Cristini.
\newblock Nonlinear modelling of cancer: bridging the gap between cells and
  tumours.
\newblock {\em Nonlinearity}, 23(1):R1--R91, 2010.

\bibitem{VaVi2009}
Peng Lu, Lei Ni, Juan-Luis V\'{a}zquez, and C\'{e}dric Villani.
\newblock Local {A}ronson-{B}\'{e}nilan estimates and entropy formulae for
  porous medium and fast diffusion equations on manifolds.
\newblock {\em J. Math. Pures Appl. (9)}, 91(1):1--19, 2009.

\bibitem{MRS1}
Bertrand Maury, Aude Roudneff-Chupin, and Filippo Santambrogio.
\newblock A macroscopic crowd motion model of gradient flow type.
\newblock {\em Math. Models Methods Appl. Sci.}, 20(10):1787--1821, 2010.

\bibitem{MRS2}
Bertrand Maury, Aude Roudneff-Chupin, and Filippo Santambrogio.
\newblock Congestion-driven dendritic growth.
\newblock {\em Discrete Contin. Dyn. Syst.}, 34(4):1575--1604, 2014.

\bibitem{MRSV}
Bertrand Maury, Aude Roudneff-Chupin, Filippo Santambrogio, and Juliette Venel.
\newblock Handling congestion in crowd motion modeling.
\newblock {\em Netw. Heterog. Media}, 6(3):485--519, 2011.

\bibitem{MePeQu}
Antoine Mellet, Beno\^{i}t Perthame, and Fernando Quir\'{o}s.
\newblock A {H}ele-{S}haw problem for tumor growth.
\newblock {\em J. Funct. Anal.}, 273(10):3061--3093, 2017.

\bibitem{MeSa20162}
Alp\'{a}r~Rich\'{a}rd M\'{e}sz\'{a}ros and Filippo Santambrogio.
\newblock Advection-diffusion equations with density constraints.
\newblock {\em Anal. PDE}, 9(3):615--644, 2016.

\bibitem{MoPeu}
Sebastien Motsch and Diane Peurichard.
\newblock From short-range repulsion to {H}ele-{S}haw problem in a model of
  tumor growth.
\newblock {\em J. Math. Biol.}, 76(1-2):205--234, 2018.

\bibitem{Perthame2014b}
Beno{\^\i}t Perthame, Fernando Quir{\'o}s, Min Tang, and Nicolas Vauchelet.
\newblock Derivation of a hele-shaw type system from a cell model with active
  motion.
\newblock {\em Interfaces and Free Boundaries}, 16:489--508, 2014.

\bibitem{Perthame2014}
Beno\^{i}t Perthame, Fernando Quir\'{o}s, and Juan~Luis V\'{a}zquez.
\newblock The {H}ele-{S}haw asymptotics for mechanical models of tumor growth.
\newblock {\em Archive for Rational Mechanics and Analysis}, 212(1):93--127,
  2014.

\bibitem{Preziosi2012}
L~Preziosi.
\newblock A review of mathematical models for the formation of vascular
  networks.
\newblock {\em Journal of Theoretial Biology}, 2012.

\bibitem{RJPJ}
J.~Ranft, M.~Basana, J.~Elgeti, J.-F. Joanny, J.; Prost, and F.~J\"ulicher.
\newblock Fluidization of tissues by cell division and apoptosis.
\newblock {\em Natl. Acad. Sci. USA}, 49:657--687, 2010.

\bibitem{Roose:07}
T.~Roose, S.J. Chapman, and P.K. Maini.
\newblock Mathematical models of avascular tumour growth: a review.
\newblock {\em SIAM Review}, 49(2):179--208, 2007.

\bibitem{Vazquez2007}
Juan~Luis V{\'a}zquez.
\newblock {\em The porous medium equation: mathematical theory}.
\newblock Oxford University Press, 2007.

\end{thebibliography}
\bibliographystyle{plain}

\end{document}